\documentclass[11pt]{article}

\usepackage[hmargin={1.0in,1.0in},%                   % for use with package titling and twopage option
            vmargin={1.0in,1.0in},%
            headheight=14pt,%                           % allow header using package fancyhdr
            twoside,%
            ]{geometry}

\usepackage{amsmath}                      % number the equation by section in amsmath
\usepackage{amssymb}                                    % mathematical symbols and AMS fonts
\usepackage{amsmath,amssymb,amsthm,amsfonts}
\usepackage{amsthm}
\usepackage{mathrsfs}
\usepackage[dvips]{graphicx}
\usepackage[OT2,T1]{fontenc}
\DeclareSymbolFont{cyrletters}{OT2}{wncyr}{m}{n}
\DeclareMathSymbol{\Sha}{\mathalpha}{cyrletters}{"58}
\newtheorem{theorem}{Theorem}[section]
\newtheorem{lemma}[theorem]{Lemma}
\newtheorem{corollary}[theorem]{Corollary}
\newtheorem{proposition}[theorem]{Proposition}
\theoremstyle{definition}
\newtheorem{definition}[theorem]{Definition}

\theoremstyle{remark}

\numberwithin{equation}{section}
\newcommand{\C}{\mathbb{C}}

\newcommand{\R}{\mathbb{R}}

\newcommand{\ra}{\rightarrow}
\newcommand{\il}{\langle}
\newcommand{\ir}{\rangle}
\begin{document}
\title{Sharp estimates for the commutators of the Hilbert, Riesz transforms and the Beurling-Ahlfors operator on weighted Lebesgue spaces}
\author{Daewon Chung}
\date{\today}
\maketitle

\begin{abstract}
We prove that the operator norm on weighted Lebesgue space
$L^2(w)$ of the commutators of the Hilbert, Riesz and Beurling
transforms with a $BMO$ function $b$ depends quadratically on the
$A_2$-characteristic of the weight, as opposed to the linear
dependence known to hold for the operators themselves. It is known
that the operator norms of these commutators can be controlled by
the norm of the commutator with appropriate Haar shift operators,
and we prove the estimate for these commutators. For the shift
operator corresponding to the Hilbert transform we use Bellman
function methods, however there is now a general theorem for a
class of Haar shift operators that can be used instead to deduce
similar results. We invoke this general theorem to obtain the
corresponding result for the Riesz transforms and the
Beurling-Ahlfors operator. We can then extrapolate to $L^p(w)$,
and the results are sharp for $1<p<\infty$.
\end{abstract}

\section{Introduction}
We say $w$ is a weight if it is positive almost everywhere and
locally integrable. The norm of $f\in L^p(w)$ is \begin{equation*}
\|f\|_{L^p(w)}:=\bigg(\int_{\mathbb{R}}|f(x)|^pw(x)dx\bigg)^{1/p}\,.\end{equation*}
Helson and Szeg\"{o} first found necessary and sufficient conditions for the boundedness of the Hilbert
transform
\begin{equation*}
Hf(x)=p.v.\frac{1}{\pi}\int\frac{f(y)}{x-y}dy\end{equation*}
 in weighted Lebesgue spaces in \cite{HS}. Hunt,
Muckenhoupt, and Wheeden showed in \cite{HMW} that a new necessary
and sufficient condition for boundedness of the Hilbert transform
in $L^p(w)$ is that the weight satisfies the $A_p$ condition,
namely: \begin{equation}\label{In:Ap}[w]_{A_p}:=\sup_I\langle
w\rangle_I\langle
w^{-1/(p-1)}\rangle^{p-1}_I<\infty\,,\end{equation} where we
denote the average over the interval $I$ by
$\langle\cdot\rangle_I\,,$ and we take the supremum over all
intervals in $\R\,.$ After one year, Coifman and Fefferman
extended in \cite{CF} the result to a larger class of convolution
singular integrals with standard kernels. Recently, many authors
have been interested in finding the sharp bounds for the operator
norms in terms of the $A_p$-characteristic $[w]_{A_p}$ of the
weight. That is, one looks for a function $\phi(x)$, sharp in
terms of its growth, such that
\begin{equation*}\|Tf\|_{L^p(w)}\leq
C\phi\big([w]_{A_p}\big)\|f\|_{L^p(w)}\,.\end{equation*} For
$T=M,$ the Hardy-Littlewood maximal function, S. Buckley \cite{SB}
showed that $\phi(x)=x^{1/(p-1)}$ is the sharp rate of growth for
all $1<p<\infty\,.$ He also showed in \cite{SB} that $\phi(x)=x^2$
works for the Hilbert transform in $L^2(w)\,.$ S. Petermichl and
S. Pott improved the result to $\phi(x)=x^{3/2}\,,$ for the
Hilbert transform in $L^2(w)\,$ in \cite{PP1}. More recently, S.
Petermichl proved in \cite{SP2} the linear dependence,
$\phi(x)=x,$ for the Hilbert transform in $L^2(w)\,$
\begin{equation*} \|Hf\|_{L^2(w)}\leq
C[w]_{A_2}\|f\|_{L^2(w)}\,,\end{equation*} by estimating the
operator norm of the dyadic shift. Linear bounds in $L^2(w)$ were
also obtained by O. Beznosova for the dyadic paraproduct
\cite{OB}. Most recently there are new proofs for the linear
estimates in $L^2(w)$ for some operators including
 the Hilbert transform and the dyadic paraproduct, \cite{LPS} and
 \cite{CMP}. The conjecture is that all the Calder\'{o}n-Zygmund
 singular integral operator obey linear bounds in weighted
 $L^2\,.$ So far this is known only for the Beurling-Ahlfors operator \cite{DV}, \cite{PV},
 the Hilbert transform \cite{SP2}, Riesz transforms \cite{SP3}, the martingale transform \cite{WI},
 the square function \cite{HTV}, \cite{CP} and well localized dyadic
 operators \cite{ML}, \cite{LPS}, \cite{CMP}. It is now also known for Calder\'{o}n-Zygmund convolutions operators that are
 smooth averages of well localized operators \cite{AV}.\\

In this paper, we are interested in obtaining sharp weight
inequalities for the commutators of the Hilbert, Riesz, and
Beurling transforms with multiplication by a locally integrable
function $b\in BMO\,.$
\subsection{Commutators: main results}
 Commutator operators are widely encountered
and studied in many problems in PDEs,
  and Harmonic Analysis. One classical result of Coifman, Rochberg, and Weiss states in \cite{CRW}
  that, for the Calder\'{o}n-Zygmund singular integral operator
  with a smooth kernel,
  $[b,T]f:=bT(f)-T(bf)$ is a bounded operator on $L^p_{\mathbb{R}^n}$, $1<p<\infty\,,$ when $b$ is a BMO
  function. Weighted estimates for the commutator have been studied in \cite {ABKP}, \cite{P1}, \cite{P2}, and \cite{P3}.
  Note that the commutator $[b,T]$ is more singular than the associated singular integral operator $T$, in particular,
  it does not satisfy the corresponding weak $(1,1)$ estimate. However one can find a weaker estimate in \cite{P2}.
  In 1997, C. P\'{e}rez \cite{P2} obtained the following result concerning commutators of singular integrals, for $1<p<\infty\,,$
  $$\|[b,T]f\|_{L^p(w)}\leq C\|b\|_{BMO}[w]_{A_{\infty}}^2\|M^2f\|_{L^p(w)}\,,$$ where $M^2=M\circ M$ denotes
  the Hardy-Littlewood maximal function iterated twice. With this result
and Buckley's sharp estimate for the maximal function \cite{SB}
one can immediately conclude that
  $$\|[b,T]\|_{L^p(w)\ra L^p(w)}\leq C [w]^2_{A_{\infty}}[w]_{A_p}^{\frac{2}{p-1}} \|b\|_{BMO}\,.$$
In this paper we show that for $T$ the Hilbert, Riesz, Beurling
transform, for $1<p\leq 2$ one can drop $[w]_{A_{\infty}}$ term,
in the above estimate, and this is sharp. However for $p>2\,,$ the
$L^p(w)$-norm of $[b,T]$ is bounded above by
$\|b\|_{BMO^d}[w]^2_{A_p^2}\,.$ For $T=H$ the Hilbert transform we
prove, using Bellman function techniques similar to those used in
\cite{OB}, \cite{SP2}, the following Theorem.
\begin{theorem}\label{Pr:Mr}
There exists a constant $C>0$, such that
$$ \| [b,H]f \|_{L^p(w)\ra L^p(w)} \leq C \|b\|_{BMO} [w]_{A_p}^{2\max\{1,\frac{1}{p-1}\}}\|f\|_{L^p(w)}\,,$$
and this is sharp for $1<p<\infty$.
\end{theorem} Most of the work goes into showing the quadratic
estimate for $p=2$, sharp extrapolation \cite{DGPP} then provides
the right rate of growth for $p\neq 2$. An example of C. P\'{e}rez
shows the rates are sharp. Our method involves the use of the
dyadic paraproduct $\pi_b$ and its adjoint $\pi_b^*$, both of
which obey linear estimates in $L^2(w)$, see \cite{OB}, like the
Hilbert transform. It also uses Petermichl description of the
Hilbert transform as an average of dyadic shift operators $S$,
\cite{SP1} and reduces the estimate to obtaining corresponding
estimates for the commutator $[b,S]$. After we decompose this
commutator in three parts:
$$[b,S]f=[\pi_b,S]f+[\pi^{\ast}_b,S]f+[\lambda_b,S]f\,,$$
we estimate each commutator separately. This decomposition has
been used before to analyze the commutator, \cite{SP1}, \cite{ML},
\cite{LPPW}. For precise definitions and detail derivations, see
Section 1.2. The first two commutators immediately give the
desired quadratic estimates in $L^2(w)$ from the known linear
bounds of the operators commuted. For the third commutator we can
prove a better than quadratic bound, in fact a linear bound. The
following Theorem will be the crucial part of the proof.
\begin{theorem}\label{DS:mc}
There exists a constant $C>0$, such that
\begin{equation}\label{DS:mt}\|[\lambda_b,S]\|_{L^2(w)\to
L^2(w)}\leq C[w]_{A^d_2}\|b\|_{BMO^{\,d}}\,,\end{equation} for all
$b\in BMO^{\,d}$ and $w\in A^d_2$.
\end{theorem}
This theorem is an immediate consequence of results in
\cite{HLRV}, \cite{LPS} and \cite{CMP}, since the operator
$[\lambda_b,S]$ belongs to the class of Haar shift operators for
which they can prove linear bounds. We present a different proof
of this result and others, using Bellman function techniques and
bilinear Carleson embedding theorems, very much in the spirit of
\cite{SP1} and \cite{OB}. These arguments were found independently
by the author, and we think they can be
of interest.\\

We then observe that for any Haar shift operator $T$ as defined in
\cite{LPS} the commutator $[\lambda_b, T]$ is again a Haar shift
operator, and therefore it obeys linear bounds in
$A_2$-characteristic of the weight as in Theorem \ref{DS:mc}. As a
consequence, we obtain quadratic bounds for all commutators of
Haar shift operators and BMO function $b\,.$ In particular, this
holds true for Haar shift operators in $\R^n$ whose averages
recover the Riesz transforms \cite{SP3} and for martingale
transforms in $\R^2$ whose averages recover the Beurling-Ahlfors
operator \cite{PV}, \cite{DV}. Extrapolation will provide $L^p(w)$
bounds which turn out to be sharp for the Riesz transforms and
Beurling-Ahlfors operators as well. The following Theorem holds
\begin{theorem}\label{ExCo} Let $T_{\tau}$ be the first class of Haar shift operators of index $\tau\,.$ Its convex hull include the Hilbert transform,
 Riesz transforms, the Beurling-Ahlfors operator and so on. Then, there exists a constant $C(\tau,n,p)$ which only depend
  on $\tau\,,$ $n$ and $p$ such that
$$\|[b,T_{\tau}]\|_{L^p(w)\ra L^p(w)}\leq C(\tau ,n ,p)[w]_{A_p}^{2\max\{1,\frac{1}{p-1}\}}\|b\|_{BMO}$$
\end{theorem}
See Section 10 for definitions and precise statements of these
results.
\subsection{The Hilbert transform case} The bilinear operator $$fH(g)+H(f)g$$
maps $L^2\times L^2$ into $H^1$, here $H$ is the Hilbert transform
and $H^1$ is the real Hardy space defined by
$$H^1(\R):=\{f\in L^1(\R):\,Hf\in L^1(\R)\}$$
with norm $$\|f\|_{H^1}=\|f\|_{L^1}+\|Hf\|_{L^1}\,.$$ Because the dual of $H^1$ is $BMO$, we will
pair with a $BMO$ function $b\,.$ Using that $H^{\ast}=-H$, we obtain that $$\il\,fH(g)+H(f)g,\,b\,\ir=\il\, f,\, H(g)b-H(gb)\,\ir\,.$$
Hence the operator $g\mapsto H(g)b-H(gb)$ should be $L^2$ bounded. This expression $H(g)b-H(gb)$ is called the commutator of $H$ with the $BMO$ function $b\,.$ More generally, we define as follows.
\begin{definition}
The commutator of the Hilbert transform $H$ with a function $b$ is defined as $$[b,H](f)=bH(f)-H(bf)\,.$$
\end{definition}
Our main concern in this paper is to prove that the commutator
$[b,H]$, as an operator from $L^2_{\R}(w)$ into $L^2_{\R}(w)$ is
bounded by the square of the $A_2$-characteristic, $[w]_{A_2}\,,$
of the weight times the $BMO$ norm, of $b\,,$ $\|b\|_{BMO}\,,$
where
$$\|b\|_{BMO}:=\sup_{I}\frac{1}{|\,I|}\int_I|\,b(x)-\il
b\ir_I|dx\,.$$ The supremum is taken over all intervals in $\R\,.$
Note that when we restrict the supremum to dyadic intervals this
will define $BMO^{\,d}$ and we denote this dyadic $BMO$ norm by
$\|\cdot\|_{BMO^{\,d}}\,.$ We now state our main results.
\begin{theorem}\label{Mr:t1}There exists $C$ such that for all $w\in A_2\,,$ $$\|[b,H]\|_{L^2(w)\ra L^2(w)}\leq C[w]^2_{A_2}\|b\|_{BMO}\,,$$
for all $b\in BMO\,.$
\end{theorem}
Once we have boundedness and sharpness for the crucial case $p=2$,
we can carry out the power of the $A_p$-characteristic, for any
$1<p<\infty\,,$ using the sharp extrapolation theorem \cite{DGPP}
to obtain Theorem \ref{Pr:Mr}. Furthermore, an example of C.
P\'{e}rez \cite{PP} shows this quadratic power is sharp.  In
\cite{SP1}, S. Petermichl showed that the norm of the commutator
of the Hilbert transform is bounded by the supremum of the norms
of the commutator of certain shift operators. This result follows
after writing the kernel of the Hilbert transform as a well chosen
averages of certain dyadic shift operators discovered by
Petermichl. More precisely, S. Petermichl showed there is a non
zero constant $C$ such that
\begin{equation} \|[b,H]\|\leq
C\sup_{\alpha,r}\|[b,S^{\alpha,r}]\|\,,\end{equation} where the
dyadic shift operator $S^{\alpha, r}$ is defined by $$S^{\alpha,
r}f=\sum_{I\in\mathcal{D}^{\alpha, r}}\il f
,h_I\ir(h_{I_-}-h_{I_+})\,.$$ Where $h_I$ denotes the Haar
function associated to the interval $I$, $I_{\pm}$ denote the left
and right halves of $I$, and $\mathcal{D}^{\alpha,r}$ is a shifted
and dilated system of dyadic intervals. Denote by $S$ the shift
operator corresponding to the standard dyadic intervals
$\mathcal{D}\,.$ See the next section
for a precise definition.\\

Let us consider a compactly supported $b\in BMO^{\,d}$ and $f\in
L^2\,.$ Expanding  $b$ and $f$ in the Haar system associated to
the dyadic intervals $\mathcal{D}\,,$
$$b(x)=\sum_{I\in\mathcal{D}}\il b,h_I\ir h_I(x),~~f(x)=\sum_{J\in\mathcal{D}}\il f,h_J\ir h_J(x)\,;$$
formally, we get the multiplication of $b$ and $f$ to be broken into three terms,
\begin{equation}bf=\pi^{\ast}_b(f)+\pi_b(f)+\lambda_b(f)\,,\label{Ma:bf}\end{equation}
where $\pi_b$ is the dyadic paraproduct, $\pi^{\ast}_b$ is its
adjoint and $\lambda_b(\cdot)=\pi_{(\cdot)}b\,,$ defined as
follows
\begin{equation*}\pi^*_b(f)(x):=\sum_{I\in\mathcal{D}}\il b,h_I\ir\il f,h_I\ir h^2_I(x)\,,\end{equation*}
\begin{equation*}\pi_b(f)(x):=\sum_{I\in\mathcal{D}}\il b,h_I\ir \il f\ir_I h_I(x)\,,\end{equation*}
\begin{equation*}\lambda_b(f)(x):=\sum_{I\in\mathcal{D}}\il b\ir_I\il f,h_I\ir  h_I(x)\,.\end{equation*}
Thus, we have
\begin{equation}[b,S]=[\pi^*_b,S]+[\pi_b,S]+[\lambda_b,S]\,,\label{Ma:sp}\end{equation}
where $$S(f)=\sum_{I\in\mathcal{D}}\langle f,h_I\rangle
(h_{I_-}-h_{I_+})$$ and we can estimate each term separately.
Notice that both $\pi_b$ and $\pi^{\ast}_b$ are bounded operators
in $L^p$ for $b\in BMO\,,$ despite the fact that multipication by
$b$ is not a bounded operator in $L^2$ unless $b$ is bounded
$(L^{\infty})\,.$ Therefore, $\lambda_b$ can not be a bounded
operator in $L^p\,.$ However $[\lambda_b,S]$ will be bounded on
$L^p(w)$ and will be better behaved than $[b,S]\,.$ Decomposition
(\ref{Ma:sp}) was used to analyze the commutator with the shift
operator first by Petermichl in \cite{SP1}, but also Lacey in
\cite{ML} and authors in \cite{LPPW} to analyze the iterated
commutators. Since all estimates are independent on the dyadic
grid, through out this paper we only deal with the dyadic shift
operator $S$ associated to the standard dyadic grid
$\mathcal{D}\,.$ For a single shift operator the hypothesis
required on $b$ and $w$ are that they belong to dyadic $BMO^{\,d}$
and $A^d_2$ with respect to the underlying dyadic grid defining
the operator. However since ultimately we want to average over all
grids, we will need $b$ and $w$ belonging to $BMO^{\,d}$ and
$A^d_2$ for all shifted and scaled dyadic grids, that we will have
if $b\in BMO$ and $w\in A_2\,,$ non-dyadic $BMO$ and $A_2\,.$
Beznosova has proved linear bounds for $\pi_b$ and
$\pi^{\ast}_b\,,$ \cite{OB}, together with Petermichl's \cite{SP2}
linear bounds for $S\,,$ this immediately provides the quadratic
bounds for $[\pi_b,S]$ and $[\pi_b^{\ast},S]$. Theorem \ref{Mr:t1}
will be proved once we show the quadratic estimate holds for
$[\lambda_b,S]$. We can actually obtain a better linear estimate
as in Theorem \ref{DS:mc}. Some terms in (\ref{Ma:sp}) do also
obey linear bounds.

\begin{theorem}\label{Lb:S1}
There exists $C$ such that
$$\|\pi^{\ast}_bS\|_{L^2(w)}+\|S\pi_b\|_{L^2(w)}\leq C[w]_{A^d_2}\|b\|_{BMO^d}\,.$$ for all $b\in BMO^d\,.$
\end{theorem}
Note the three operators $[\lambda_b,S]\,,$ $\pi^{\ast}_bS$ and
$S\pi_b$ are generalized Haar shift operators for which there are
now two different proofs of linear bounds on $L^2(w)$ with respect
to $[w]_{A^d_2}\,,$ \cite{LPS} and \cite{CMP}, and in this paper
we present a third proof. We are now ready to explain the
organization of this paper. In Section 2 we will introduce
notation and discuss some useful results about weighted Haar
systems. In Section 3 we will start our discussion on how to find
the linear bound for the term $[\lambda_b,S]\,,$ most of which
will be very similar to calculations performed in \cite{SP2}. In
Section 4 we will introduce a number of Lemmas and Theorems that
will be used. In Section 5 we will finish the linear estimate for
the term $[\lambda_b,S]\,,$ and prove Theorem \ref{Mr:t1}. In
Section 6 we prove the linear bound for $\pi_b^*S\,.$ In Section 7
we reduce the proof of the linear bound for $S\pi_b$ to verifying
three embedding conditions, two are proved in this section, the
third is proved in Section 8 using a Bellman function argument. In
Section 9 we present the $L^p(w)$ estimate of the commutator with
a Haar shift operator, Theorem \ref{ExCo}. Finally, in Section 10
we provide the sharpness for the commutators of Hilbert, Riesz
transforms and Beurling-Ahlfors operators.

\section{Preliminaries and Notation}
Let us now introduce the notation which will be used frequently through this paper. Even though the $A_p$ conditions have already been introduced in (\ref{In:Ap}), we will state the special case of this condition when $p=2\,,$ namely $A^d_2$  since we will refer repeatedly to this. We say $w$ belongs to $A^d_2$ class, if
\begin{equation}\label{In:A2}[w]_{A^d_2}:=\sup_{I\in\mathcal{D}}\langle w\rangle_I\langle w^{-1}\rangle_I<\infty\,.\end{equation}
Here we take the supremum over all dyadic interval in
$\mathbb{R}\,.$ Note that if $w\in A_2$ then $w\in A_2^d$ and
$[w]_{A^d_2}\leq [w]_{A_2}\,.$ Intervals of the form
$[k2^{-j},(k+1)2^{-j})$ for integers $j,k$ are called dyadic
intervals. Again, let us denote $\mathcal{D}$ the collection of
all dyadic intervals, and let us denote $\mathcal{D}(J)$ the
collection of all dyadic subintervals of $J\,.$ For any interval
$I\in\mathcal{D}$, there is a Haar function defined by
\begin{equation*}h_I(x)=\frac{1}{|I|^{1/2}}(\chi_{I_+}(x)-\chi_{I_-}(x))\,,\end{equation*}
where $\chi_I$ denotes the characteristic function of the interval
$I\,,$ $\chi_I(x)=1$ if $x\in I\,,$ $\chi_I(x)=0$ otherwise. It is
a well known fact that the Haar system $\{h_I\}_{I\in\mathcal{D}}$
is an orthonormal system in $L^2_{\mathbb{R}}\,.$ We also consider
the different grids of dyadic intervals parametrized by
$\alpha,\,r,$ defined by
\begin{equation*}\mathcal{D}^{\,\alpha, r}=\{\alpha+rI:
I\in\mathcal{D}\}\,,\end{equation*} for $\alpha\in\R$ and positive
$r\,.$ For each grid $\mathcal{D}^{\,\alpha,r}$ of dyadic
intervals, there are corresponding Haar functions
$h_I^{\alpha,r}\,,$ $I\in \mathcal{D}^{\,\alpha,r}$ that are an
orthonormal system in $L^2_{\mathbb{R}}\,.$ Let us introduce a
proper orthonormal system for $L^2_{\mathbb{R}}(w)$ defined by
\begin{equation*}h^w_I:=\frac{1}{w(I)^{1/2}}\bigg[\frac{w(I_-)^{1/2}}{w(I_+)^{1/2}}\chi_{I_+}-\frac{w(I_+)^{1/2}}{w(I_-)^{1/2}}\chi_{I_-}\bigg]\,,\end{equation*}
where $w(I)=\int_Iw\,.$ We define the weighted inner product by
$\langle f,g\rangle_w=\int_{\mathbb{R}}fgw\,.$ Then, every
function $f\in L^{2}(w)$ can be written as
$$f=\sum_{I\in\mathcal{D}}\langle f,h^w_I\rangle_wh^w_{I}\,,$$ where the
sum converges a.e. in $L^2(w)\,.$ Moreover,
\begin{equation}\label{In:l2n}\|f\|^2_{L^2(w)}=\sum_{I\in\mathcal{D}}|\langle
f,h^w_I\rangle_{w}|^2\,.\end{equation} Again $\mathcal{D}$ can be
replaced by $\mathcal{D}^{\,\alpha,r}$ and the corresponding
weighted Haar functions are an orthonormal system in $L^2(w)\,.$
For convenience we will observe basic properties of the
disbalanced Haar system. First observe that $\il h_K,h^w_I\ir_w$
could be non-zero only if $I\supseteq K$, moreover, for any
$I\supseteq K$,
 \begin{equation}\label{Pr:p1}|\il h_K,h^w_I\ir_w|\leq\il
w\ir_K^{1/2}\,.\end{equation} Here is the the calculation that
provides (\ref{Pr:p1}),
\begin{align*}
|\il h_K,h^w_I\ir_w|&=\bigg|\int\frac{1}{|K|^{1/2}w(I)^{1/2}}(\chi_{K_+}(x)-\chi_{K_-}(x))\bigg[\frac{w(I_-)^{1/2}}{w(I_+)^{1/2}}\chi_{I_+}(x)-\frac{w(I_+)^{1/2}}{w(I_-)^{1/2}}\chi_{I_-}(x)\bigg]w(x)dx\bigg|\\
&\leq\frac{1}{|K|^{1/2}w(I)^{1/2}}\underbrace{\int_K\bigg|\frac{w(I_-)^{1/2}}{w(I_+)^{1/2}}\chi_{I_+}(x)+\frac{w(I_+)^{1/2}}{w(I_-)^{1/2}}\chi_{I_-}(x)\bigg|w(x)dx}_{A}\,.
\end{align*}
If $K\subset I_+$, then $A\leq w(I_-)^{1/2}w(I_+)^{-1/2}w(K)$.
Thus $$|\il h_K,h^w_I\ir_w|\leq \il
w\ir_K^{1/2}\sqrt{\frac{w(K)w(I_-)}{w(I_+)w(I)}}\leq \il
w\ir_K^{1/2}\,.$$ Similarly, if $K\subset I_-\,.$ If $K=I$, then
$A\leq 2\,w(K_-)^{1/2}w(K_+)^{1/2}\,.$ Thus
$$|\il h_K,h^w_I\ir_w|=|\il h_K,h^w_K\ir_w|\leq \il w\ir_K^{1/2}\,\frac{2\sqrt{w(K_-)w(K_+)}}{w(K)}\leq\il w\ir_K^{1/2}\,.$$
Estimate (\ref{Pr:p1}) implies that $|\il
h_{\hat{J}},h_{\hat{J}}^{w^{-1}}\ir_{w^{-1}}\il
h_J,h_J^w\ir_w|\leq \sqrt{2}[w]_{A_2^d}^{1/2}\,,$ where $\hat{J}$
is the parent of $J\,,$
\begin{align}
|\il h_{\hat{J}},h_{\hat{J}}^{w^{-1}}\ir_{w^{-1}}\il h_J,h_J^w\ir_w|&\leq \il w^{-1}\ir_{\hat{J}}^{1/2}\il w\ir_{J}^{1/2}=\il w^{-1}\ir_{\hat{J}}^{1/2}\bigg(\frac{1}{|J|}\int_Jw\bigg)^{1/2}\nonumber\\
&=\il w^{-1}\ir_{\hat{J}}^{1/2}\bigg(\frac{2}{|\hat{J}|}\int_Jw\bigg)^{1/2}\leq\sqrt{2}\il w^{-1}\ir_{\hat{J}}^{1/2}\il w\ir_{\hat{J}}^{1/2}\nonumber\\
&\leq \sqrt{2}\,[w]_{A^d_2}^{1/2}\,.\label{Pr:hh}
\end{align}
Also, one can deduce similarly the following estimate
\begin{equation}|\il h_{J},h_{J}^{w^{-1}}\ir_{w^{-1}}\il
h_{J_-},h_{J}^w\ir_w|\leq
\sqrt{2}[w]_{A_2^d}^{1/2}\,.\label{Pr:p7}\end{equation} For
$I\supsetneq J$, $h^w_I$ is constant on $J$. We will denote this
constant by $h^w_I(J)\,.$ Then $h^w_{\hat{J}}(J)$ is the constant
value of $h^w_{\hat{J}}$ on $J$ and $|h^w_{\hat{J}}(J)|\leq
w(J)^{-1/2}\,,$ as can be seen by the following estimate,
\begin{equation}\label{Pr:p9}|h^w_{\hat{J}}(J)|=\left\{\begin{array}{l}
                                          w(\hat{J}_-)^{1/2}/w(\hat{J})^{1/2}w(\hat{J}_+)^{1/2}\leq w(J)^{-1/2}~~\textrm{if}~~J=\hat{J}_+\\
                                          w(\hat{J}_+)^{1/2}/w(\hat{J})^{1/2}w(\hat{J}_-)^{1/2}\leq
                                          w(J)^{-1/2}~~\textrm{if}~~J=\hat{J}_-\,.
                           \end{array}\right.\end{equation}

Let us define the weighted averages, $\langle
g\rangle_{J,w}:=w(J)^{-1}\int_Jg(x)w(x)dx\,.$ As with the standard
Haar system, we can write the weighted averages
\begin{equation}\il g\ir_{J,w}=\sum_{I\in\mathcal{D}:I\supsetneq
J}\il g,h^w_I\ir_wh^w_I(J)\,.\label{Pr:e4}\end{equation} In fact,
here is the derivation of (\ref{Pr:e4}).
\begin{align*}
\il g\ir_{J,w}&=\frac{1}{w(J)}\int_J\sum_{I\in\mathcal{D}}\il
g,h^w_I\ir_wh^w_I(x)w(x)dx=\frac{1}{w(J)}\int_J\sum_{I\in\mathcal{D}}\il
g,h^w_I\ir_wh^w_I(J)w(x)dx\nonumber\\
&=\frac{1}{w(J)}\int_Jw(x)dx\sum_{I\in\mathcal{D}:I\supsetneq J}\il
g,h^w_I\ir_wh^w_I(J)=\sum_{I\in\mathcal{D}:I\supsetneq J}\il
g,h^w_I\ir_wh^w_I(J)\,.
\end{align*} Also, we will be using system of functions $\{H^w_I\}_{I\in\mathcal{D}}$ defined by
\begin{equation}
H^w_I=h_I\sqrt{|I|}-A^w_I\chi_I~~\textrm{where}~~A^w_I=\frac{\langle w\rangle_{I_+}-\langle w\rangle_{I_-}}{2\langle w\rangle_I}\,.\label{Pr:HA}\end{equation}
Then, $\{w^{1/2}H^w_I\}$ is orthogonal in $L^2$ with norms satisfying the inequality $\|w^{1/2}H^w_I\|_{L^2}\leq \sqrt{|I|\langle w\rangle_{I}}\,,$ refer to \cite{OB}.
Moreover, by Bessel's inequality we have, for all $g\in L^2\,,$ \begin{equation}\sum_{I\in \mathcal{D}}\frac{1}{|\,I|\langle w\rangle_I}\langle g,w^{1/2} H^w_I\rangle ^2\leq\|g\|^2_{L^2}\,.\label{In:e4}\end{equation}

\section{Linear bound for $[\lambda_b,S]$ part I}
In general, when we analyse commutator operators, a subtle
cancelation delivers the result one wants to find. In the analysis
of the commutator $[b,S]$, the part $[\lambda_b,S]$ will allow for
certain cancelation. First, let us rewrite $[\lambda_b,S]\,.$
\begin{align*}
[\lambda_b,S](f)&=\lambda_b(Sf)-S(\lambda_bf)\\
&=\sum_{I\in\mathcal{D}}\langle b\rangle_I\langle Sf,h_I\rangle h_I-\sum_{J\in\mathcal{D}}\langle \lambda_b f,h_J\rangle(h_{J_-}-h_{J_+})\\
&= \sum_{I\in\mathcal{D}}\sum_{J\in\mathcal{D}}\langle b\rangle_I\langle f,h_J\rangle\langle h_{J_-}-h_{J_+},h_I\rangle h_I-\sum_{J\in\mathcal{D}}\sum_{I\in\mathcal{D}}\langle b\rangle_I \langle f,h_I\rangle\langle h_I,h_J\rangle(h_{J_-}-h_{J_+})\\
&=\sum_{J\in\mathcal{D}}\langle b\rangle_{J_-}\langle f,h_J\rangle h_{J_-}-\sum_{J\in\mathcal{D}}\langle b\rangle_{J_+}\langle f,h_J\rangle h_{J_+}-\sum_{J\in\mathcal{D}}\frac{\langle b\rangle_{J_+}+\langle b\rangle_{J_-}}{2}\langle f,h_J\rangle(h_{J_-}-h_{J_+})\\
&=\sum_{J\in\mathcal{D}}\frac{\langle b\rangle_{J_-}-\langle b\rangle_{J_+}}{2}\langle f,h_J\rangle h_{J_-}-\sum_{J\in\mathcal{D}}\frac{\langle b\rangle_{J_+}-\langle b\rangle_{J_-}}{2}\langle f,h_J\rangle h_{J_+}\\
&=-\sum_{J\in\mathcal{D}}\Delta _J b\langle f,h_J\rangle(h_{J_+}+h_{J_-})\,,
\end{align*}
recall the notation $\Delta_J\,b=(\langle b\rangle_{J_+}-\langle
b\rangle_{J_-})/2\,.$ To find the $L^2(w)$ operator norm of
$[\lambda_b,S]$, it is enough to deal with the operator
$$S_b(f)=\sum_{I\in\mathcal{D}}\Delta_I\,b\langle f,h_I\rangle
h_{I_-}\,.$$ We shall state the weighted operator norm of $S_b$ as
a Theorem and give a detailed proof. Theorem \ref{DS:mc} is a
direct consequence of the following Theorem. We will prove the
following Theorem by the technique used in \cite{SP2}.
\begin{theorem}\label{DS:ma}
There exists a constant $C>0$, such that \begin{equation}\label{DS:sh}\|S_b\|_{L^2(w)\ra L^2(w)}\leq C[w]_{A^d_2}\|b\|_{BMO^{\,d}}\end{equation}
for all $b\in BMO^{\,d}$ and $w\in A^d_2$ for all $f\in L^2(w).$
\end{theorem}
One can easily check that choosing $b=|I|^{1/2}h_I\,,$ yields
$\|S\|_{L^2(w)\ra L^2(w)}\leq C[w]_{A^d_2}\,.$ Inequality (\ref{DS:sh}) is equivalent to the following inequality for any
positive function $f,g$
\begin{equation}\label{DS:shd}|\il S_{b,w^{-1}} f,g\ir_w|\leq C[w]_{A^d_2}\|b\|_{BMO^{\,d}}\|f\|_{L^2(w^{-1})}\|g\|_{L^2(w)}\,,\end{equation}
where $S_{b,w^{-1}}(f)=S_b(w^{-1}f)\,.$ Expanding $f$ and $g$ in
the disbalanced Haar systems respectively for $L^2(w^{-1})$ and
$L^2(w)$ yields for (\ref{DS:shd}),
\begin{align}
|\il S_{b,w^{-1}}f,g\ir_w|&=\bigg|\int S_{b,w^{-1}}\big(\sum_{I\in\mathcal{D}}\il f,h_I^{w^{-1}}\ir_{w^{-1}}h^{w^{-1}}_I\big)\big(\sum_{J\in\mathcal{D}}\il g,h^w_J\ir_wh^w_J\big)wdx\bigg|\nonumber\\
&=\bigg|\sum_{I\in\mathcal{D}}\sum_{J\in\mathcal{D}}\il f,h_I^{w^{-1}}\ir_{w^{-1}}\il g,h^w_J\ir_w\int h^w_J S_{b,w^{-1}}(h^{w^{-1}}_I)wdx\bigg|\nonumber\\
&=\bigg|\sum_{I\in\mathcal{D}}\sum_{J\in\mathcal{D}}\il f,h_I^{w^{-1}}\ir_{w^{-1}}\il g,h^w_J\ir_w\il
S_{b,w^{-1}}h^{w^{-1}}_I,h^w_J\ir_w\bigg|\,.\label{DS:esh}
\end{align}
In (\ref{DS:esh}),
$$\il S_{b,w^{-1}}h^{w^{-1}}_I,h^w_J\ir_w=\Big\il\sum_{L\in\mathcal{D}}\Delta_Lb\il
h_L,w^{-1}h^{w^{-1}}_I\ir\,
h_{L^-},h^w_J\Big\ir_w=\sum_{L\in\mathcal{D}}\Delta_Lb \il
h_L,h^{w^{-1}}_I\ir_{w^{-1}}\il h_{L^-},h_J^w\ir_w$$ Since $\il
h_L,h^{w^{-1}}_I\ir_{w^{-1}}\neq 0,$ only when $L\subseteq I$ and
$\il h_{L^-},h_J^w\ir_w\neq 0$ only when $L_-\subseteq J\,,$ then
we have non-zero terms if $I\subseteq J$ or $\hat{J}\subseteq I\,$
in the sum of (\ref{DS:esh}). Thus we can split the sum into four
parts, $\sum_{I=J}\,,~\sum_{I=\hat{J}}\,,~\sum_{\hat{J}\subsetneq
I}\,,$ and $\sum_{I\subsetneq J}\,.$ Let us now introduce the
truncated shift operator $$
S^{I}_{\,b}(f):=\sum_{L\in\mathcal{D}(I)}\Delta_Lb\il f, h_L\ir
h_{L_-}\,,$$ and its composition with multiplication by $w^{-1}$,
$$
S^{I}_{\,b,w^{-1}}(f):=\sum_{L\in\mathcal{D}(I)}\Delta_Lb\langle
w^{-1}f,h_L\rangle h_{L_-}\,.$$ We will see that the weighted norm
$\| S^{I}_{\,b,w^{-1}}\chi_I\|_{L^2(w)}$, plays
 a main role in our estimate for $\il S_{\,b,w^{-1}}h^{w^{-1}}_I,h^w_J\ir_w\,.$

\section{Theorems and Lemmas} To prove inequality (\ref{DS:shd}) we
need several theorems and lemmas. One can find the proof in the
indicated references. First we recall some embedding theorems.
Another main result in \cite{SP2} is a two-weighted bilinear
embedding theorem, which was proved by a Bellman function
argument.
\begin{theorem}[Petermichl's Bilinear Embedding Theorem]\label{BIT}  Let $w$ and $v$ be weights so that $\il w\ir_I\il v\ir_I\leq Q$ for all intervals $I$ and let $\{\alpha_I\}$ be a non-negative sequence so that the three estimates below hold for all $J$
\begin{equation}\label{BIT1}\sum_{I\in\mathcal{D}(J)}\frac{\alpha_I}{\il w\ir_I}\leq Q\,v(J)\end{equation}
\begin{equation}\label{BIT2}\sum_{I\in\mathcal{D}(J)}\frac{\alpha_I}{\il v\ir_I}\leq Q\,w(J)\end{equation}
\begin{equation}\label{BIT3}\sum_{I\in\mathcal{D}(J)}\alpha_I\leq Q\,|J|\,.\end{equation}
Then there is $c$ such that for all $f\in L^2(w)$ and $g\in L^2(v)$ $$\sum_{I\in\mathcal{D}}\alpha_I\il f\ir_{I,w}\il g\ir_{I,v}\leq cQ\|f\|_{L^2(w)}\|g\|_{L^2(v)}\,.$$
\end{theorem}
\begin{corollary}[Bilinear Embedding Theorem] Let $w$ and $v$ be weights. Let $\{\alpha_I\}$ be a sequence of nonnegative numbers such that for all dyadic intervals $J\in\mathcal{D}$ the following three inequalities hold with some constant $Q>0\,,$
\begin{equation}\label{BITb}\sum_{I\in\mathcal{D}(J)}\alpha_I\il v\ir_I\,|\,I|\leq Q\,v(J)\end{equation}
\begin{equation}\label{BITa}\sum_{I\in\mathcal{D}(J)}\alpha_I\il w\ir_I\,|\,I|\leq Q\,w(J)\end{equation}
\begin{equation}\label{BITc}\sum_{I\in\mathcal{D}(J)}\alpha_I\il w\ir_I\il v\ir_I\leq Q\,|J|\,.\end{equation}
Then for any two nonnegative function $f,\,g\in L^2$ $$\sum_{I\in\mathcal{D}}\alpha_I\langle fv^{1/2}\rangle_I\langle gw^{1/2}\rangle_I\,|\,I|\leq CQ\|f\|_{L^2}\|g\|_{L^2}$$ holds with some constant $C>0\,.$
\end{corollary}  Both Bilinear Embedding Theorems are key tools in our estimate. One version of such a theorem appeared in \cite{NTV}.
The original version of the next lemma also appeared in
\cite{SP2}. Using the fact that, for all $I\in\mathcal{D}\,,$
$|\Delta_Ib|\leq 2\|b\|_{BMO^d}\,,$ one can prove the following
lemma similarly to its original proof.
\begin{lemma}\label{TS}
There is a constant $c$ such that
$\|S^I_{b,w^{-1}}\chi_I\|_{L^2(w)}\leq
c\|b\|_{BMO^{\,d}}[w]_{A^d_2}w^{-1}(I)^{1/2}$ for all intervals $I\,$
and weights $w\in A^d_2\,.$
\end{lemma}
We will also need the Weighted Carleson Embedding Theorem from \cite{NTV}, and some other inequalities for weights.
\begin{theorem}[Weighted Carleson Embedding Theorem]\label{CIT}
 Let $\{\alpha_J\}$ be a non-negative sequence such that for all dyadic intervals $I$
$$\sum_{J\in\mathcal{D}(I)}\alpha_J\leq Qw^{-1}(I)\,.$$
Then for all $f\in L^2(w^{-1})$
$$\sum_{J\in\mathcal{D}}\alpha_J\il f\ir^2_{J,w^{-1}}\leq 4Q\|f\|^2_{L^2(w^{-1})}\,.$$
\end{theorem}

\begin{theorem}[Wittwer's sharp version of Buckley's inequality]\label{WSB}
There exist a positive constant $C$ such that for any weight $w\in A^d_2$ and dyadic interval $I\in\mathcal{D}\,,$
$$\frac{1}{|J|}\sum_{I\in\mathcal{D}(J)}\frac{\big(\langle w\rangle_{I_+}-\langle w\rangle_{I_-}\big)^2}{\langle w\rangle_I}\,|\,I|\leq C[w]_{A^d_2}\langle w\rangle_J\,.$$
\end{theorem}
We refer to \cite{WI} for the proof. You can find extended versions of Theorem \ref{CIT} and \ref{WSB} to the doubling positive measure $\sigma$ in \cite{CP}. One can find the Bellman function proof of the following three Lemmas in \cite{OB}.
\begin{lemma}\label{icl1}
For all dyadic interval J and all weights $w\,.$
$$\frac{1}{|J|}\sum_{I\in\mathcal{D}(J)}|I||\Delta_Iw|^2\frac{1}{\langle w\rangle_I^3}\leq \langle w^{-1}\rangle_J\,.$$
\end{lemma}

\begin{lemma}\label{icl2}
Let $w$ be a weight and $\{\alpha_I\}$ be a Carleson sequence of nonnegative numbers. If there exist a constant $Q>0$ such that $$\forall\,J\in\mathcal{D}\,,~~\frac{1}{|J|}\sum_{I\in\mathcal{D}(J)}\alpha_I\leq Q\,,$$then
$$\forall\,J\in\mathcal{D}\,,~~\frac{1}{|J|}\sum_{I\in\mathcal{D}(J)}\frac{\alpha_I}{\langle w^{-1}\rangle_I}\leq 4Q\langle w\rangle _J$$
and therefore if $w\in A^d_2$ the for any $J\in\mathcal{D}$ we have
$$\frac{1}{|J|}\sum_{I\in\mathcal{D}(J)}\langle w\rangle _I\alpha_I\leq 4 Q[w]_{A^d_2}\langle w\rangle _J\,.$$
\end{lemma}

\begin{lemma}\label{icl3}
If $w\in A^d_2$ then there exists a constant $C>0$ such that  $$\forall J \in\mathcal{D}\,,~~\frac{1}{|J|}\sum_{I\in\mathcal{D}(J)}\bigg(\frac{\langle w\rangle _{I_+}-\langle w\rangle_{I_-}}{\langle w\rangle_I}\bigg)^2\,|\,I|\,\langle w\rangle_I\langle w^{-1}\rangle_{I}\leq C[w]_{A^d_2}\,.$$
\end{lemma}

\section{Linear bound for $[\lambda_b,S]$ part II}
We will continue to estimate the sum (\ref{DS:esh}) in four parts.
\subsection{$\sum_{I=\hat{J}}$}\label{DSII1}

For this case, it is sufficient to show that $$|\il
S_{b,w^{-1}}h^{w^{-1}}_{\hat{J}},h^w_J\ir_w|\leq
c\|b\|_{BMO^{\,d}}[w]_{A^d_2}\,.$$ Since $\il h_k,h^w_I\ir_w$
could be non-zero only if $K\subseteq I$,
$$|\il S_{b,w^{-1}}h^{w^{-1}}_{\hat{J}},h^w_J\ir_w|=\Big|\Big\il\sum_{L\in\mathcal{D}}\Delta_{L}b\,\il w^{-1}h^{w^{-1}}_{L},h_L\ir h_{L_-},h^w_J\Big\ir_w\Big|=\Big|\sum_{L\in\mathcal{D}}\Delta_{L}b\,\il h^{w^{-1}}_{\hat{J}},h_L\ir_{w^{-1}}\il h_{L_-},h^w_J\ir_w\Big|\,$$
has non-zero term only when $L\subseteq \hat{J}$. Thus
\begin{align*}&|\il S_{b,w^{-1}}h^{w^-1}_{L},h^w_J\ir_w|=\Big|\sum_{L\in\mathcal{D}(\hat{J}\,)}\Delta_{L}b\,\il h^{w^{-1}}_{\hat{J}},h_L\ir_{w^{-1}}\il h_{L_-},h^w_J\ir_w\Big|\\
&~~~~~~~~~~~~=\Big|\sum_{L\in\mathcal{D}(J)}\Delta_{L}b\,\il h^{w^{-1}}_{\hat{J}},h_L\ir_{w^{-1}}\il h_{L_-},h^w_J\ir_w\Big|+\Big|\sum_{L\in\mathcal{D}( J^s)}\Delta_{L}b\,\il h^{w^{-1}}_{\hat{J}},h_L\ir_{w^{-1}}\il h_{L_-},h^w_J\ir_w\Big|\\
&~~~~~~~~~~~~~~~+|\Delta_{\hat{J}}\,b\,\il h^{w^{-1}}_{\hat{J}},h_{\hat{J}}\ir_{w^{-1}}\il h_{\hat{J}_-},h^w_J\ir_w|\\
&~~~~~~~~~~~~\leq |\il
S^{J}_{b,w^{-1}}h^{w^{-1}}_{\hat{J}},h^w_J\ir_w|+|\Delta_{\hat{J}}\,b\,\il
h^{w^{-1}}_{\hat{J}},h_{\hat{J}}\ir_{w^{-1}}\il
h_{\hat{J}_-},h^w_J\ir_w|
\end{align*}
in the second equality, $J^s$ denotes the sibling of $J$, so for all $L\subseteq J^s$, $\il h_{L_-},h^w_J\ir_w=0\,.$ Then, by (\ref{Pr:hh}),
\begin{equation}\label{DSII:e1}|\Delta_{\hat{J}}\,b\,\il h^{w^{-1}}_{\hat{J}},h_{\hat{J}}\ir_{w^{-1}}\il h_{\hat{J}_-},h^w_J\ir_w|\leq \sqrt{2}\|b\|_{BMO^{\,d}}[w]^{1/2}_{A_2^d}\,.\end{equation}
So for the remaining part:
\begin{equation}\label{DSII:e2}|\il S^{J}_{b,w^{-1}}h^{w^{-1}}_{\hat{J}},h^w_J\ir_w|=|h^{w^{-1}}_{\hat{J}}(J)\il S^J_{b,w^{-1}}\chi_J,h^w_J\ir_w|\leq c\|b\|_{BMO^{\,d}}[w]_{A^d_2}\,,\end{equation}
here the last inequality uses (\ref{Pr:p9}) and Lemma \ref{TS}.

\subsection{$\sum_{I=J}$}

In this case, the argument is similar to the argument in Section \ref{DSII1}. We have
$$|\il S_{b,w^{-1}}h^{w^{-1}}_{J},h^w_J\ir_w|=|\sum_{L\in\mathcal{D}}\Delta_{L}b\,\il h^{w^{-1}}_{J},h_{L}\ir_{w^{-1}}\il h_{L_-},h^w_J\ir_w|$$
here we have zero summands, unless $L\subseteq J$. Thus,
\begin{align*}
&|\il S_{b,w^{-1}}h^{w^{-1}}_{J},h^w_J\ir_w|=|\il S^J_{b,w^{-1}}h^{w^{-1}}_{J},h^w_J\ir_w|\\
&~~~~~~~~~~~~\leq\, |\il S^{J_+}_{b,w^{-1}}h^{w^{-1}}_{J},h^w_J\ir_w|+|\il S^{J_-}_{b,w^{-1}}h^{w^{-1}}_{J},h^w_J\ir_w|+|\Delta_{J}\,b\,\il h^{w^{-1}}_{J},h_{J}\ir_{w^{-1}}\il h_{J_-},h^w_J\ir_w|\\
&~~~~~~~~~~~~\leq\,c\,\|b\|_{BMO^{\,d}}[w]_{A_2^d}\,.
\end{align*}
In the last inequality, we use same arguments as in (\ref{DSII:e2}) for the first two term, and (\ref{Pr:p7}) for the last term.\\

\subsection{$\sum_{\hat{J}\subsetneq I}$  and $\sum_{I\subsetneq J}$}\label{DSII3}

To obtain our desired results, we need to understand the supports of $
S_{b}(w^{-1}h^{w^{-1}}_I)$ and $ S^{\ast}_b(wh^w_J)\,.$ Since
$$ S_{b}(w^{-1}h^{w^{-1}}_I)=\sum_{L\in\mathcal{D}}\Delta_Lb\,\il
w^{-1}h^{w^{-1}}_I,h_L\ir
h_{L_-}=\sum_{L\in\mathcal{D}}\Delta_Lb\,\il
h^{w^{-1}}_I,h_L\ir_{w^{-1}} h_{L_-}\,,$$ and $\langle h_I^{w^{-1}},h_L\rangle$ can be non-zero only when $L\subseteq I\,,$  $
S_{b}(w^{-1}h^{w^{-1}}_I)$ is supported by $I\,.$ Also,
\begin{equation}\il S_b(w^{-1}h^{w^{-1}}_I),h^w_J\ir_w=\il h^{w^{-1}}_I,
S^{\ast}_b(wh^w_J)\ir_{w^{-1}}\,,\label{DS2:e1}
\end{equation}
yield  that $ S^{\ast}_b(wh^w_J)=\sum_{L\in\mathcal{D}}\Delta_Lb\langle wh_J^w,h_{L_-}\rangle h_L$ is supported by $\hat{J}\,.$ Let us now consider the sum $\hat{J}\subsetneq I\,.$ Then
\begin{align}
&\bigg|\sum_{I,J:\hat{J}\subsetneq I}\il f,h^{w^{-1}}_I\ir_{w^{-1}}\il g,h^w_J\ir_w\il S_{b,w^{-1}}h^{w^{-1}}_I,h^w_J\ir_w\bigg|\nonumber\\
&~~~~~~~~~~~~=\,\bigg|\sum_{I,J:\hat{J}\subsetneq I}\il f,h^{w^{-1}}_I\ir_{w^{-1}}\il g,h^w_J\ir_w\il h^{w^{-1}}_I, S^{\ast}_{b}(wh^w_J)\ir_{w^{-1}}\bigg|\nonumber\\
&~~~~~~~~~~~~=\,\bigg|\sum_{J\in\mathcal{D}}\sum_{I:I\supsetneq\hat{J}}\il f,h^{w^{-1}}_I\ir_{w^{-1}}\il g,h^w_J\ir_wh^{w^{-1}}_I(\hat{J})\il S_{b,w^{-1}}\chi_{\hat{J}},h^w_J\ir_w\bigg|\label{DS2:e2}\\
&~~~~~~~~~~~~=\,\bigg|\sum_{J\in\mathcal{D}}\il f\ir_{\hat{J},w^{-1}} \il g,h^w_J\ir_w\il S_{b,w^{-1}}\chi_{\hat{J}},h^w_J\ir_w\bigg|\label{DS2:e3}\\
&~~~~~~~~~~~~\leq \,\|g\|_{L^2(w)}\bigg(\sum_{J\in\mathcal{D}}\il f\ir^2_{\hat{J},w^{-1}}\il
S_{b,w^{-1}}\chi_{\hat{J}},h^w_J\ir^2_w\bigg)^{1/2}\,,\label{DS2:e4}
\end{align}
here (\ref{DS2:e1}) and the fact that $ S^{\ast}_b(wh^w_J)$ is
supported by $\hat{J}$ are used for equality (\ref{DS2:e2}), and
(\ref{DS2:e3}) uses (\ref{Pr:e4}) and (\ref{DS2:e1}). If we show
that \begin{equation}\label{DS2:e5}\sum_{J\in\mathcal{D}}\il
f\ir^2_{\hat{J},w^{-1}}\il
S_{b,w^{-1}}\chi_{\hat{J}},h^w_J\ir^2_w\leq
c\|b\|^2_{BMO^{\,d}}[w]_{A_2^d}^2\|f\|^2_{L^2(w^{-1})}\,,\end{equation}
then we have $$ \Bigg|\sum_{I,J:\hat{J}\subsetneq I}\il
f,h^{w^{-1}}_I\ir_{w^{-1}}\il g,h^w_J\ir_w\il
S_{b,w^{-1}}h^{w^{-1}}_I,h^w_J\ir_w\Bigg|\leq
C\|b\|_{BMO^{\,d}}[w]_{A_2^d}\|f\|^2_{L(w^{-1})}\|g\|_{L^2(w)}\,.$$ To
prove the inequality (\ref{DS2:e5}), we apply the Theorem
\ref{CIT}. The embedding condition becomes
$$\sum_{J\in\mathcal{D}:J\subsetneq I}\il S_{b,w^{-1}}\chi_{\hat{J}},h^w_J\ir^2_w\leq c\|b\|^2_{BMO^{\,d}}[w]_{A_2^d}^2w^{-1}(I)\,$$ after shifting the indices. Since $\il h_{L_-},h^w_J\ir_w=0$ unless $L\subseteq\hat{J}\,,$  and we will sum over $J$ such that $J\subsetneq I$, we can write
\begin{align*}
\il S_{b,w^{-1}}\chi_{\hat{J}},h^w_J\ir_w&=\sum_{L\in\mathcal{D}}\Delta_Lb\il w^{-1}\chi_{\hat{J}},h_L\ir\il h_{L_-},h^w_J\ir_w=\sum_{L\in\mathcal{D}(\hat{J}\,)}\Delta_Lb\il w^{-1}\chi_{\hat{J}},h_L\ir\il h_{L_-},h^w_J\ir_w\\
&=\sum_{L\in\mathcal{D}(I\,)}\Delta_Lb\il w^{-1}\chi_{I},h_L\ir\il h_{L_-},h^w_J\ir_w=\il S^I_{b,w^{-1}}\chi_{I},h^w_J\ir_w\,.
\end{align*}
Thus,
$$\sum_{J\in\mathcal{D}:J\subsetneq I}\il S_{b,w^{-1}}\chi_{\hat{J}},h^w_J\ir^2_w=\sum_{J\in\mathcal{D}:J\subsetneq I}\il S^{I}_{b,w^{-1}}\chi_I,h^w_J\ir_w^2\leq\| S^{I}_{b,w^{-1}}\chi_I\|_{L^2(w)}^2\,,$$
last inequality due to (\ref{In:l2n}).  By Lemma \ref{TS}, the embedding condition holds. Hence we are done for the sum $\hat{J}\subsetneq I$. The part $\sum_{I\subsetneq J}$ is similar to $\sum_{\hat{J}\subsetneq I}\,.$ One
uses that $ S_{b}(w^{-1}h^{w^{-1}}_I)$ is supported by $I$ and
Theorem \ref{CIT}.

\subsection{ Proof of Theorem \ref{Mr:t1}}
We refer to \cite{OB} for following theorem.
\begin{theorem}\label{lbp}
The norm of dyadic paraproduct on the weighted Lebesgue space
$L^2(w)$ is bounded from above by a constant multiple of the
product of the $A^d_2$ characteristic of the weight $w$ and the
$BMO^{\,d}$ norm of $b$.
\end{theorem}

To break $[b,S]$ into three parts, as in (\ref{Ma:sp}), we assumed
that $b\in BMO^{\,d}$ is compactly supported. However, we need to
replace such a $b$ with  a general $BMO^{\,d}$ function. In order
to pass from a compactly supported $b$ to general $b\in
BMO^{\,d}\,,$ we need the following lemma which is suggested in
\cite{JBG}\,.
\begin{lemma}\label{GJB}
Suppose $\phi\in BMO\,.$ Let $\widetilde{I}$ be the interval
concentric with $I$ having length $|\widetilde{I}|=3|I|\,.$ Then
there is $\psi\in BMO$ such that $\psi=\phi$ on $I$, $\psi=0$ on
$\R\setminus\widetilde{I}$ and $\|\psi\|_{BMO}\leq
c\|\phi\|_{BMO}\,.$
\end{lemma}
\begin{proof}
Without loss of generality, we assume $\langle\phi\rangle_I=0\,.$ Write $I=\bigcup^{\,\infty}_{\,n=0}J_n$ where $dist(J_n,\partial I)=|J_n|$, as in following figure.
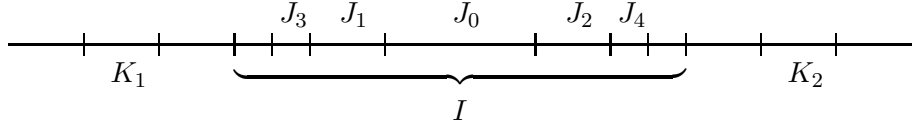
\begin{center}
\setlength{\unitlength}{1mm}
\begin{picture}(120,20)
\put(0,10){\line(1,0){120}}\put(30,8.5){\line(0,1){3}}\put(90,8.5){\line(0,1){3}}\put(50,8.5){\line(0,1){3}}\put(70,8.5){\line(0,1){3}}
\put(40,8.5){\line(0,1){3}}\put(80,8.5){\line(0,1){3}}\put(35,8.5){\line(0,1){3}}\put(85,8.5){\line(0,1){3}}
\put(59,13){$J_0$}\put(44,13){$J_1$}\put(74,13){$J_2$}\put(36,13){$J_3$}\put(81,13){$J_4$}
\put(30,7){$\underbrace{~~~~~~~~~~~~~~~~~~~~~~~~~~~~~~~~~~~~~~~~~~~~~~~}$}\put(59,0){$I$}
\put(20,8.5){\line(0,1){3}}\put(10,8.5){\line(0,1){3}}\put(100,8.5){\line(0,1){3}}\put(110,8.5){\line(0,1){3}}
\put(13.5,5){$K_1$}\put(103.5,5){$K_2$}
\end{picture}
\end{center}
Then $J_0$ is the middle third of $I\,.$ For $n>0$, let $K_n$ be the reflection of $J_n$ across the nearest endpoint of $I$ and set
$$\psi(x)=\left\{\begin{array}{l}
                                          \phi(x)~~\textrm{if}~~x\in I\\
                                          \langle \phi\rangle_{J_n}~~\textrm{if}~~x\in K_n\\
                                          0~~\textrm{otherwise}\,.
                           \end{array}\right.$$
This construction of $\psi$ satisfies Lemma \ref{GJB}.

\end{proof}

By  Theorem \ref{DS:ma}, Corollary \ref{DS:mc}, Theorem \ref{lbp}, Lemma \ref{GJB}
and using the fact $\|\pi_b\|=\|\pi^{\ast}_b\|$, we can prove
Theorem \ref{Mr:t1}.
\begin{proof}[Proof of Theorem \ref{Mr:t1}]
For any compactly supported $b\in BMO\,,$
\begin{align*}
\|[b,H]\|_{L^2(w)\ra L^2(w)}\leq&\,C\sup_{\alpha,r}\,\|[b,S^{\,\alpha,r}]\|_{L^2(w)\ra L^2(w)}\\
\leq&\,C\sup_{\alpha,r}\big(\,\|[\pi_b,S^{\,\alpha,r}]\|_{L^2(w)\ra
L^2(w)}+\|[\pi^{\ast}_b,S^{\,\alpha,r}]\|_{L^2(w)\ra
L^2(w)}+\|[\lambda_b,S^{\,\alpha,r}]\|_{L^2(w)\ra L^2(w)}\big)\\
\leq&\,C\big(4\|\pi_b\|_{L^2(w)\ra
L^2(w)}\sup_{\alpha,r}\|S^{\,\alpha,r}\|_{L^2(w)\ra
L^2(w)}+C[w]_{A_2}\|b\|_{BMO}\,\big)\\
\leq&C[w]^2_{A_2}\|b\|_{BMO}\,.
\end{align*}
For fixed $b$, we consider the sequence of intervals $I_k=[-k,k]$
and the sequence of $BMO$ functions $b_k$ which are constructed as
in Lemma \ref{GJB}. Then, there is a constant $c$, which does not
depend on $k\,,$ such that $\|b_k\|_{BMO}\leq
c\|b\|_{BMO}\,.$ Furthermore, there is a uniform constant
$C$ such that
\begin{equation}\|[b_k,H]\|_{L^2(w)\ra L^2(w)}\leq C[w]^2_{A_2}\|b\|_{BMO}\,.\label{DS:bm}\end{equation} Therefore, for some subsequence of integers $k_j$ and $f\in L^2(w)\,,$ $[b_{k_j},H](f)$ converges to $[b,H](f)$ almost everywhere. Letting $j\ra\infty$ and using Fatou's lemma, we deduce that (\ref{DS:bm}) holds for all $b\in BMO\,.$
\end{proof}

\section{Linear bound for $\pi^{\ast}_bS$}

It might be useful to know what is the adjoint operator of $S$.
Let us define sgn$(I)=\pm1$ if $I=\hat{I}_{\mp}\,.$ Then, for any function $f,g\in L^2\,,$
\begin{align*}
\langle Sf,g\rangle&=\sum_{I\in\mathcal{D}}\sum_{J\in\mathcal{D}}\langle f,h_I\rangle\langle g,h_J\rangle\langle h_{I_-}-h_{I_+},h_J\rangle\\
&=\sum_{I\in\mathcal{D}}\langle f,h_I\rangle\langle g,h_{I_-}\rangle-\sum_{I\in\mathcal{D}}\langle f,h_I\rangle\langle g,h_{I_+}\rangle\\
&=\sum_{I\in\mathcal{D}}\langle f,h_I\rangle \big(\textrm{sgn}(I_-)\langle g,h_{I_-}\rangle+\textrm{sgn}(I_+)\langle g,h_{I_+}\rangle\big)\\
&=\sum_{I\in\mathcal{D}}\langle f,h_{\hat{I}}\rangle \textrm{sgn}(I)\langle g,h_I\rangle\\
&=\Big\langle
f,\sum_{I\in\mathcal{D}}\textrm{sgn}(I)\langle g,h_I\rangle
h_{\hat{I}}\Big\rangle=\langle f, S^{\ast}g\rangle\,.
\end{align*}
Now, we see the adjoint operator of dyadic shift operator $S$ is
$$S^{\ast}f(x)=\sum_{I\in\mathcal{D}}\textrm{sgn}(I)\langle
f,h_I\rangle h_{\hat{I}}(x)\,.$$ The following lemma provides the bound we are looking for the term $\pi^{\ast}_bS\,.$
\begin{lemma}\label{gcs}
Let $w\in A^d_2$ and $b\in BMO^{\,d}$. Then, there exists $C$ so that $$\|\pi^{\ast}_bS\|_{L^2(w)\ra L^2(w)}\leq C[w]_{A^d_2}\|b\|_{BMO^{\,d}}\,.$$
\end{lemma}
\begin{proof} In order to prove Lemma \ref{gcs} it is enough to show that for any positive square integrable function $f,g$
\begin{equation}\langle \pi^{\ast}_bS(fw^{-1/2}),gw^{1/2}\rangle\leq C[w]_{A^d_2}\|b\|_{BMO^{\,d}}\|f\|_{L^2}\|g\|_{L^2}\,.\label{GB:p0}\end{equation}
Using the system of functions $\{H^w_I\}_{I\in\mathcal{D}}\,$
defined in (\ref{Pr:HA}), we can rewrite the left hand side of
(\ref{GB:p0})
\begin{align}
\langle \pi^{\ast}_bS(fw^{-1/2}),gw^{1/2}\rangle&=\langle S(fw^{-1/2}),\pi_b(gw^{1/2})\rangle=\sum_{I\in\mathcal{D}}\langle gw^{1/2}\rangle_I\langle b,h_I\rangle\langle S(fw^{-1/2}),h_I\rangle\nonumber\\
&=\sum_{I\in\mathcal{D}}\langle gw^{1/2}\rangle_I\langle b,h_I\rangle\,\textrm{sgn}(I)\langle fw^{-1/2},h_{\hat{I}}\rangle\nonumber\\
&=\sum_{I\in\mathcal{D}}\,\textrm{sgn}(I)\langle gw^{1/2}\rangle_I\langle b,h_I\rangle\langle fw^{-1/2},H^{w^{-1}}_{\hat{I}}\rangle\frac{1}{\sqrt{|\hat{I}|}}\nonumber\\
&~~~~~~~~~~~+\sum_{I\in\mathcal{D}}\,\textrm{sgn}(I)\langle gw^{1/2}\rangle_I\langle b,h_I\rangle\langle fw^{-1/2},A^{w^{-1}}_{\hat{I}}\chi_{\hat{I}}\rangle\frac{1}{\sqrt{|\hat{I}|}}\,.\label{GB:p1}
\end{align}
Our claim is that both sums in (\ref{GB:p1}) are bounded by $[w]_{A^d_2}\|b\|_{BMO^{\,d}}\|f\|_{L^2}\|g\|_{L^2}\,,$ i.e.
\begin{equation}\sum_{I\in\mathcal{D}}\,\textrm{sgn}(I)\langle gw^{1/2}\rangle_I\langle b,h_I\rangle\langle fw^{-1/2},H^{w^{-1}}_{\hat{I}}\rangle\frac{1}{\sqrt{|\hat{I}|}}\leq C[w]_{A^d_2}\|b\|_{BMO^{\,d}}\|f\|_{L^2}\|g\|_{L^2}\label{GB:e1}\end{equation}
and
\begin{equation}\sum_{I\in\mathcal{D}}\,\textrm{sgn}(I)\langle gw^{1/2}\rangle_I\langle b,h_I\rangle\langle fw^{-1/2},A^{w^{-1}}_{\hat{I}}\chi_{\hat{I}}\rangle\frac{1}{\sqrt{|\hat{I}|}}\leq C[w]_{A^d_2}\|b\|_{BMO^{\,d}}\|f\|_{L^2}\|g\|_{L^2}\,.\label{GB:e2}\end{equation}
First let us verify the bound for (\ref{GB:e1}). Using Cauchy-Schwarz inequality,
\begin{align}
&\sum_{I\in\mathcal{D}}\,\textrm{sgn}(I)\langle gw^{1/2}\rangle_I\langle b,h_I\rangle\langle fw^{-1/2},H^{w^{-1}}_{\hat{I}}\rangle\frac{1}{\sqrt{|\hat{I}|}}\nonumber\\
&~~~~~~~~~~~~\leq\bigg(\sum_{I\in\mathcal{D}}\langle gw^{1/2}\rangle_I^2\langle b,h_I\rangle^2\langle w^{-1}\rangle_{\hat{I}}\bigg)^{1/2}\bigg(\sum_{I\in\mathcal{D}}\frac{1}{|\hat{I}|\langle w^{-1}\rangle_{\hat{I}}}\langle f,w^{-1/2}H^{w^{-1}}_{\hat{I}}\rangle^2\bigg)^{1/2}\nonumber\\
&~~~~~~~~~~~~\leq \|f\|_{L^2}\bigg(\sum_{I\in\mathcal{D}}\langle gw^{1/2}\rangle_I^2\langle b,h_I\rangle^2\langle w^{-1}\rangle_{\hat{I}}\bigg)^{1/2}\,.\label{GB:e3}
\end{align}
Thus, for (\ref{GB:e1}), it is enough to show that
\begin{equation}
\sum_{I\in\mathcal{D}}\langle gw^{1/2}\rangle_I^2\langle b,h_I\rangle^2\langle w^{-1}\rangle_{\hat{I}}\leq C[w]^2_{A^d_2}\|b\|^2_{BMO^{\,d}}\|g\|^2_{L^2}\,.\label{GB:e4}
\end{equation}
It is clear that $2\langle w\rangle_{\hat{I}}\geq \langle w\rangle_I\,,$ thus
\begin{align*}
\sum_{I\in\mathcal{D}}\langle gw^{1/2}\rangle_I^2\langle b,h_I\rangle^2\langle w^{-1}\rangle_{\hat{I}}&=\sum_{I\in\mathcal{D}}\langle gw^{1/2}\rangle_I^2\langle b,h_I\rangle^2\langle w^{-1}\rangle_{\hat{I}}\langle w\rangle_I\langle w\rangle_I^{-1}\\
&\leq 2[w]_{A^d_2}\sum_{I\in\mathcal{D}}\langle gw^{1/2}\rangle_I^2\langle b,h_I\rangle^2\langle w\rangle_I^{-1}\,.
\end{align*}
If we show for all $J\in\mathcal{D}\,,$
\begin{equation}\frac{1}{|J|}\sum_{I\in\mathcal{D}(J)}\langle b,h_I\rangle^2\langle w\rangle_I^{-1}\langle w\rangle_I^2=\frac{1}{|J|}\sum_{I\in\mathcal{D}(J)}\langle b,h_I\rangle^2\langle w\rangle_I\leq [w]_{A^d_2}\|b\|^{2}_{BMO^{\,d}}\langle w\rangle_J\,,\label{GB:e5}\end{equation}
then by Weighted Carleson Embedding Theorem \ref{CIT} with $w$ instead of $w^{-1}\,,$ we will have (\ref{GB:e4})\,.
Since $b\in BMO^{\,d}\,,$ $\{\langle b,h_I\rangle^2\}_{I\in\mathcal{D}}$ is a Carleson sequence with constant $\|b\|^2_{BMO^{\,d}}$ that is
$$\frac{1}{|J|}\sum_{I\in\mathcal{D}(J)}\langle b,h_I\rangle^2\leq \|b\|^2_{BMO^{\,d}}\,.$$ Applying Lemma \ref{icl2} with $\alpha_I=\langle b,h_I\rangle$ we have inequality (\ref{GB:e5}). We now concentrate on the estimate (\ref{GB:e2}), we can estimate the left hand side of (\ref{GB:e2}) as follows.
\begin{align*}&\sum_{I\in\mathcal{D}}\,\textrm{sgn}(I)\langle gw^{1/2}\rangle_I\langle b,h_I\rangle\langle fw^{-1/2},A^{w^{-1}}_{\hat{I}}\chi_{\hat{I}}\rangle\frac{1}{\sqrt{|\hat{I}|}}\\
&~~~~~~~~~~~~=\sum_{I\in\mathcal{D}}\,\textrm{sgn}(I)\langle gw^{1/2}\rangle_{I}\langle b,h_I\rangle\langle fw^{-1/2}\rangle_{\hat{I}}A^{w^{-1}}_{\hat{I}}\sqrt{|\hat{I}|}\\
&~~~~~~~~~~~~\leq\sum_{I\in\mathcal{D}}\langle gw^{1/2}\rangle_{I}\,|\langle b,h_{I}\rangle|\,\langle fw^{-1/2}\rangle_{\hat{I}}\,|A^{w^{-1}}_{\hat{I}}|\sqrt{|\hat{I}|}\\
&~~~~~~~~~~~~\leq2\sum_{I\in\mathcal{D}}\langle gw^{1/2}\rangle_{\hat{I}}\,|\langle b,h_{I}\rangle|\,\langle fw^{-1/2}\rangle_{\hat{I}}\,|A^{w^{-1}}_{\hat{I}}|\sqrt{|\hat{I}|}\\
&~~~~~~~~~~~~=2\sum_{I\in\mathcal{D}}\langle gw^{1/2}\rangle_I\,\big(|\langle b,h_{I_-}\rangle|+|\langle b,h_{I_+}\rangle|\big)\,\langle fw^{-1/2}\rangle_I\,|A_I^{w^{-1}}|\sqrt{|I|}\,.
\end{align*}
By Bilinear Embedding Theorem, inequality (\ref{GB:e2}) holds
provided the following three inequalities hold,
\begin{align}
&\forall\,J\in\mathcal{D}, \frac{1}{|J|}\sum_{I\in\mathcal{D}(J)}\big(|\langle b,h_{I_-}\rangle|+|\langle b,h_{I_+}\rangle|\big)|A_I^{w^{-1}}|\sqrt{|I|}\langle w^{-1}\rangle_I\langle w\rangle_I\leq C\|b\|_{BMO^{\,d}}\,[w^{-1}]_{A^d_2}\,,\label{GB:e6}\\
&\forall\,J\in\mathcal{D}, \frac{1}{|J|}\sum_{I\in\mathcal{D}(J)}\big(|\langle b,h_{I_-}\rangle|+|\langle b,h_{I_+}\rangle|\big)|A_I^{w^{-1}}|\sqrt{|I|}\langle w^{-1}\rangle_I\leq C\|b\|_{BMO^{\,d}}\,[w^{-1}]_{A^d_2}\langle w^{-1}\rangle_J\,,\label{GB:e7}\\
&\forall\,J\in\mathcal{D}, \frac{1}{|J|}\sum_{I\in\mathcal{D}(J)}\big(|\langle b,h_{I_-}\rangle|+|\langle b,h_{I_+}\rangle|\big)|A_I^{w^{-1}}|\sqrt{|I|}\langle w\rangle_I\leq C\|b\|_{BMO^{\,d}}\,[w^{-1}]_{A^d_2}\langle w\rangle_J\,.\label{GB:e8}
\end{align}
For (\ref{GB:e6}), by Cauchy-Schwarz inequality
\begin{align*}
&\frac{1}{|J|}\sum_{I\in\mathcal{D}(J)}\big(|\langle b,h_{I_-}\rangle|+|\langle b,h_{I_+}\rangle|\big)|A_I^{w^{-1}}|\sqrt{|I|}\langle w^{-1}\rangle_I\langle w\rangle_I\\
&\leq \bigg(\frac{1}{|J|}\sum_{I\in\mathcal{D}(J)} \big(|\langle b,h_{I_-}\rangle|+|\langle b,h_{I_+}\rangle|\big)^2\langle w^{-1}\rangle_I\langle w\rangle_I\bigg)^{1/2}\bigg(\frac{1}{|J|}\sum_{I\in\mathcal{D}(J)}(A^{w^{-1}}_I)^2|I|\langle w^{-1}\rangle_I\langle w\rangle_I\bigg)^{1/2}\,.
\end{align*}
Since \begin{equation}\sum_{I\in\mathcal{D}}(|\langle b,h_{I_-}\rangle| +|\langle b,h_{I_+}\rangle|)^2\leq 3\sum_{I\in\mathcal{D}}\langle b,h_I\rangle^2\,,\label{GE:e9}\end{equation}
$$\frac{1}{|J|}\sum_{I\in\mathcal{D}(J)}(|\langle b,h_{I_-}\rangle|+|\langle b,h_{I_+}\rangle|)^2\langle w^{-1}\rangle_I\langle w\rangle_I\leq C [w^{-1}]_{A^d_2}\frac{1}{|J|}\sum_{I\in\mathcal{D}(J)}\langle b,h_I\rangle^2\leq C[w^{-1}]_{A^d_2}\|b\|^2_{BMO^{\,d}}\,,$$
and by Lemma \ref{icl3}, $$\frac{1}{|J|}\sum_{I\in\mathcal{D}(J)}(A^{w^{-1}}_I)^2|I|\langle w^{-1}\rangle_I\langle w\rangle_I\leq C\,[w^{-1}]_{A^d_2}\,.$$
Thus embedding condition (\ref{GB:e6}) holds.
For (\ref{GB:e7}), by Cauchy-Schwarz inequality and (\ref{GE:e9}) we have
\begin{align*}
&\frac{1}{|J|}\sum_{I\in\mathcal{D}(J)}(|\langle b,h_{I_-}\rangle|+|\langle b,h_{I_+}\rangle|)\,|A_I^{w^{-1}}|\sqrt{|I|}\langle w^{-1}\rangle_I\\
&~~~~~~~~~~~~\leq C\bigg(\frac{1}{|J|}\sum_{I\in\mathcal{D}(J)}\langle b,h_I\rangle^2\langle w^{-1}\rangle_I\bigg)^{1/2}\bigg(\frac{1}{|J|}\sum_{I\in\mathcal{D}(J)}(A^{w^{-1}}_I)^2|I|\langle w^{-1}\rangle_I\bigg)^{1/2}\,.
\end{align*}
By Theorem \ref{WSB},
$$\frac{1}{|J|}\sum_{I\in\mathcal{D}(J)}(A^{w^{-1}}_I)^2|I|\langle w^{-1}\rangle_I=\frac{1}{|J|}\sum_{I\in\mathcal{D}(J)}\bigg(\frac{\langle w^{-1}\rangle_{I_+}-\langle w^{-1}\rangle_{I_-}}{2\langle w^{-1}\rangle_I}\bigg)^2|I|\langle w^{-1}\rangle_I\leq C[w^{-1}]_{A^d_2}\langle w^{-1}\rangle_J\,.$$
Similarly with (\ref{GB:e5}), we have $$\frac{1}{|J|}\sum_{I\in\mathcal{D}(J)}\langle b,h_I\rangle^2\langle w^{-1}\rangle_I\leq [w^{-1}]_{A_2}\|b\|^2_{BMO^{\,d}}\langle w^{-1}\rangle_J\,.$$
To finish, we must estimate $(\ref{GB:e8})$. In a similar way with (\ref{GB:e7}), we need to estimate
$$\bigg(\frac{1}{|J|}\sum_{I\in\mathcal{D}(J)}\langle b,h_I\rangle^2\langle w\rangle _I\bigg)^{1/2}\bigg(\frac{1}{|J|}\sum_{I\in\mathcal{D}(J)}(A^{w^{-1}}_I)^2|I|\langle w\rangle_I\bigg)^{1/2}\,.$$
By Lemma \ref{icl1}, we have
\begin{align*}
\frac{1}{|J|}\sum_{I\in\mathcal{D}(J)}(A^{w^{-1}}_I)^2|I|\langle w\rangle_I&\leq [w^{-1}]_{A^d_2}\frac{1}{|J|}\sum_{I\in\mathcal{D}(J)}(A^{w^{-1}}_I)^2|I|\langle w^{-1}\rangle_I^{-1}\nonumber\\
&=[w^{-1}]_{A^d_2}\frac{1}{|J|}\sum_{I\in\mathcal{D}(J)}\bigg(\frac{\langle w^{-1}\rangle_{I_+}-\langle w^{-1}\rangle_{I_-}}{\langle w^{-1}\rangle_I^3}\bigg)^2|I|\nonumber\\
&\leq C[w^{-1}]_{A^d_2}\langle w^{-1}\rangle_{J}\,.
\end{align*}
This completes the proof of Lemma \ref{gcs}.
\end{proof}
Due to the almost self adjoint property of the Hilbert transform,
a certain bound for $\pi_b^{\ast}H$ immediately returns the same
bound for $H\pi_b$. However we have to prove the boundedness of
$S\pi_b$ independently because $S$ is not self adjoint.
\section{Linear bound for $S\pi_b$}

\begin{lemma}\label{scp}
Let $w\in A^d_2$ and $b\in BMO^{\,d}$. Then, there exists $C$ so that $$\|S\pi_b\|_{L^2(w)\ra L^2(w)}\leq C[w]_{A^d_2}\|b\|_{BMO^{\,d}}\,.$$
\end{lemma}
\begin{proof}
We are going to prove Lemma \ref{scp} by showing
\begin{equation}
\langle S\pi_b(w^{-1}f),g\rangle_w\leq C[w]_{A^d_2}\|b\|_{BMO^{\,d}}\|f\|_{L^2(w^{-1})}\|g\|_{L^2(w)}\,,\label{SP:e1}\end{equation}
for any positive function $f,g\in L^2\,.$
Since $\langle S\pi_b(f),h_I\rangle=\textrm{sgn}(I)\langle \pi_b(f),h_{\hat{I}}\rangle=\textrm{sgn}(I)\langle f\rangle_{\hat{I}}\langle b,h_{\hat{I}}\rangle\,,$ We have
$$S\pi_b(f)=\sum_{I\in\mathcal{D}}\textrm{sgn}(I)\langle f\rangle_{\hat{I}}\langle b,h_{\hat{I}}\rangle h_I\,.$$
By expanding $g$ in the disbalanced Haar system for $L^2(w)\,,$
\begin{align*}
\langle S\pi_b(w^{-1}f),g\rangle_w&=\sum_{I\in\mathcal{D}}\langle w^{-1}f\rangle_{\hat{I}}\langle b,h_{\hat{I}}\rangle \textrm{sgn}(I)\langle h_I,g\rangle_w\\
&=\sum_{I\in\mathcal{D}}\sum_{J\in\mathcal{D}}\textrm{sgn}(I)\langle w^{-1}\rangle_{\hat{I}}\langle f\rangle_{\hat{I},w^{-1}}\langle b,h_{\hat{I}}\rangle \langle g,h^w_J\rangle_w\langle h_I,h^w_J\rangle_w\,.
\end{align*}
Since $\langle h_I,h^w_J\rangle_w$ could be non zero only if
$J\supseteq I\,,$ we can split above sum into three parts,
\begin{equation}
\sum_{I\in\mathcal{D}}\textrm{sgn}(I)\langle w^{-1}\rangle_{\hat{I}}\langle f\rangle_{\hat{I},w^{-1}}\langle b,h_{\hat{I}}\rangle \langle g,h^w_I\rangle_w\langle h_I,h^w_I\rangle_w,\label{SP:e2}
\end{equation}
\begin{equation}
\sum_{I\in\mathcal{D}}\textrm{sgn}(I)\langle w^{-1}\rangle_{\hat{I}}\langle f\rangle_{\hat{I},w^{-1}}\langle b,h_{\hat{I}}\rangle \langle g,h^w_{\hat{I}}\rangle_w\langle h_I,h^w_{\hat{I}}\rangle_w,\label{SP:ee}
\end{equation}
and \begin{equation}
\sum_{I\in\mathcal{D}}\sum_{J:J\supsetneq \hat{I}}\textrm{sgn}(I)\langle w^{-1}\rangle_{\hat{I}}\langle f\rangle_{\hat{I},w^{-1}}\langle b,h_{\hat{I}}\rangle \langle g,h^w_J\rangle_w\langle h_I,h^w_J\rangle_w\,.\label{SP:e3}\end{equation}
We claim that all sums, (\ref{SP:e2}), (\ref{SP:ee}), and (\ref{SP:e3}), can be bounded with a bound that depends on $[w]_{A^d_2}\|b\|_{BMO^{\,d}}$ at most linearly.
Since $|\langle h_I,h_I^w\rangle_w|\leq \langle w\rangle_I^{1/2}\,,$ we can estimate (\ref{SP:e2})
\begin{align*}
&\bigg|\sum_{I\in\mathcal{D}}\textrm{sgn}(I)\langle w^{-1}\rangle_{\hat{I}}\langle f\rangle_{\hat{I},w^{-1}}\langle b,h_{\hat{I}}\rangle \langle g,h^w_I\rangle_w\langle h_I,h^w_I\rangle_w\,\bigg|\\
&~~~~~~~~~~~~\leq\bigg(\sum_{I\in\mathcal{D}}\langle w^{-1}\rangle_{\hat{I}}^2\langle f\rangle_{\hat{I},w^{-1}}^2\langle b,h_{\hat{I}}\rangle^2\langle w\rangle_I\bigg)^{1/2}\bigg(\sum_{I\in\mathcal{D}}\langle g,h_I^w\rangle^2\bigg)^{1/2}\\
&~~~~~~~~~~~~\leq C\|g\|_{L^2(w)}[w]_{A^d_2}^{1/2}\bigg(\sum_{I\in\mathcal{D}}\langle f\rangle_{I,w^{-1}}^2\langle b,h_{I}\rangle^2\langle w^{-1}\rangle_{I}\bigg)^{1/2}\,.
\end{align*}
By Weighted Carleson Embedding Theorem \ref{CIT},
 $$\sum_{I\in\mathcal{D}}\langle f\rangle_{I,w^{-1}}^2\langle b,h_{I}\rangle^2\langle w^{-1}\rangle_{I}\leq C[w]_{A^d_2}\|b\|^2_{BMO^{\,d}}\|f\|^2_{L^2(w^{-1})}$$
 is provided by $$\frac{1}{|J|}\sum_{I\in\mathcal{D}(J)}\langle b,h_I\rangle^2\langle w^{-1}\rangle_I\leq [w]_{A^d_2}\|b\|_{BMO^{\,d}}^2\langle w^{-1}\rangle_J\,$$
 which we already have in (\ref{GB:e5})\,. Thus, we have
 \begin{equation}\sum_{I\in\mathcal{D}}\textrm{sgn}(I)\langle w^{-1}\rangle_{\hat{I}}\langle f\rangle_{\hat{I},w^{-1}}\langle b,h_{\hat{I}}\rangle \langle g,h^w_I\rangle_w\langle
 h_I,h^w_I\rangle_w\leq C[w]_{A^d_2}\|b\|_{BMO^{\,d}}\|f\|_{L^2(w^{-1})}\|g\|_{L^2(w)}\,.\label{SP:e4}
 \end{equation}
 Similarly to (\ref{SP:e2}), we can estimate (\ref{SP:ee}) using $|\langle h_I,h^w_{\hat{I}}\rangle_w|\leq \langle w\rangle ^{1/2}_I\leq \sqrt{2}\langle w\rangle_{\hat{I}}^{1/2}$
\begin{align*}
&\bigg|\sum_{I\in\mathcal{D}}\textrm{sgn}(I)\langle w^{-1}\rangle_{\hat{I}}\langle f\rangle_{\hat{I},w^{-1}}\langle b,h_{\hat{I}}\rangle \langle g,h^w_{\hat{I}}\rangle_w\langle h_I,h^w_{\hat{I}}\rangle_w\bigg|\\
&~~~~~~~~~~~~\leq\sqrt{2}\sum_{I\in\mathcal{D}}\langle w^{-1}\rangle_{\hat{I}}\langle f\rangle_{\hat{I},w^{-1}}\,|\langle b,h_{\hat{I}}\rangle|\, \langle g,h^w_{\hat{I}}\rangle_w\langle w\rangle^{1/2}_{\hat{I}}\\
&~~~~~~~~~~~~=2\sqrt{2}\sum_{I\in\mathcal{D}}\langle w^{-1}\rangle_{I}\langle f\rangle_{I,w^{-1}}\,|\langle b,h_{I}\rangle|\, \langle g,h^w_{I}\rangle_w\langle w\rangle^{1/2}_{I}\\
&~~~~~~~~~~~~\leq2\sqrt{2}\bigg(\sum_{I\in\mathcal{D}}\langle w^{-1}\rangle_{I}^2\langle f\rangle_{I,w^{-1}}^2\langle b,h_{I}\rangle^2\langle w\rangle_I\bigg)^{1/2}\bigg(\sum_{I\in\mathcal{D}}\langle g,h_I^w\rangle^2\bigg)^{1/2}\\
&~~~~~~~~~~~~\leq C\|g\|_{L^2(w)}[w]_{A^d_2}^{1/2}\bigg(\sum_{I\in\mathcal{D}}\langle f\rangle_{I,w^{-1}}^2\langle b,h_{I}\rangle^2\langle w^{-1}\rangle_{I}\bigg)^{1/2}\\
&~~~~~~~~~~~~\leq C[w]_{A^d_2}\|b\|_{BMO^{\,d}}\|f\|_{L^2(w^{-1})}\|g\|_{L^2(w)}\,.
\end{align*}
Since $h^w_J$ is constant on $\hat{I}$, for $J\supsetneq \hat{I}$ and we denote this constant by $h^w_J(\hat{I})\,.$ Then we know by (\ref{Pr:e4}),
$$\sum_{J:J\supsetneq \hat{I}}\langle g,h^w_J\rangle_w\langle h_I,h^w_J\rangle_w=\sum_{J:J\supsetneq \hat{I}}\langle g,h^w_J\rangle_wh^w_J(\hat{I})\langle h_I,w\rangle=\langle g\rangle_{\hat{I},w}\langle h_I,w\rangle\,.$$
Thus, we can rewrite (\ref{SP:e3})
\begin{align}
&\bigg|\sum_{I\in\mathcal{D}}\textrm{sgn}(I)\langle w^{-1}\rangle_{\hat{I}}\langle f\rangle_{\hat{I},w^{-1}}\langle b,h_{\hat{I}}\rangle \langle g\rangle_{\hat{I},w}\langle h_I,w\rangle\bigg|\nonumber\\
&~~~~~~~~~~~~\leq\sum_{I\in\mathcal{D}}\langle
w^{-1}\rangle_{\hat{I}}\langle f\rangle_{\hat{I},w^{-1}}\,|\langle
b,h_{\hat{I}}\rangle|\,\langle g\rangle_{\hat{I},w}|\langle
h_I,w\rangle|\label{SP:dc}\\
&~~~~~~~~~~~~=\sum_{I\in\mathcal{D}}\langle
w^{-1}\rangle_I\,|\langle b,h_I\rangle|\,(|\langle
h_{I_-},w\rangle|+|\langle h_{I_+},w\rangle|)\langle
f\rangle_{I,w^{-1}}\langle g\rangle_{I,w}\,.
\label{SP:e5}\end{align} We claim the sum (\ref{SP:e5}) is bounded
by
$[w]_{A^d_2}\|b\|_{BMO^{\,d}}\|f\|_{L^2(w^{-1})}\|g\|_{L^2(w)}\,.$
We are going to prove it using Petermichl's Bilinear Embedding
Theorem \ref{BIT}. Thus, we need to show that the following three
embedding conditions hold,
\begin{align}
&\forall~J\in\mathcal{D},~\frac{1}{|J|}\sum_{I\in\mathcal{D}(J)}|\langle b,h_I\rangle|\,\langle w^{-1}\rangle_I(|\langle h_{I_-},w\rangle|+|\langle h_{I_+},w\rangle|)\frac{1}{\langle w\rangle_I}\leq  C[w]_{A^d_2}\|b\|_{BMO^{\,d}}\,\langle w^{-1}\rangle_J\,,\label{SP:e6}\\
&\forall~J\in\mathcal{D},~\frac{1}{|J|}\sum_{I\in\mathcal{D}(J)}|\langle b,h_I\rangle|\,\langle w^{-1}\rangle_I(|\langle h_{I_-},w\rangle|+|\langle h_{I_+},w\rangle|)\frac{1}{\langle w^{-1}\rangle_I}\leq  C[w]_{A^d_2}\|b\|_{BMO^{\,d}}\,\langle w\rangle_J\,,\label{SP:e7}\\
&\forall~J\in\mathcal{D},~\frac{1}{|J|}\sum_{I\in\mathcal{D}(J)}|\langle
b,h_I\rangle|\,\langle w^{-1}\rangle_I(|\langle
h_{I_-},w\rangle|+|\langle h_{I_+},w\rangle|)\leq
C[w]_{A^d_2}\|b\|_{BMO^{\,d}}\,.\label{SP:e8}
\end{align}
After we split the sum in (\ref{SP:e6}):
$$\frac{1}{|J|}\sum_{I\in\mathcal{D}(J)}|\langle b,h_I\rangle|\,\langle w^{-1}\rangle_I|\langle h_{I_-},w\rangle|\frac{1}{\langle
w\rangle_I}+\frac{1}{|J|}\sum_{I\in\mathcal{D}(J)}|\langle
b,h_I\rangle|\,\langle w^{-1}\rangle_I|\langle
h_{I_+},w\rangle|\frac{1}{\langle w\rangle_I}\,,$$ we start with
Cauchy-Schwarz inequality to estimate the first sum of embedding
condition (\ref{SP:e6}),
\begin{align}
&\frac{1}{|J|}\sum_{I\in\mathcal{D}(J)}|\langle b,h_I\rangle|\,\langle w^{-1}\rangle_I|\langle h_{I_-},w\rangle|\frac{1}{\langle w\rangle_I}=\frac{1}{|J|}\sum_{I\in\mathcal{D}(J)}|\langle b,h_I\rangle|\,\langle w^{-1}\rangle_I\frac{\sqrt{|I_-|}|\Delta_{I_-}w|}{\langle w\rangle_I}\nonumber\\
&~~~~~~~~~~~~\leq\bigg(\frac{1}{|J|}\sum_{I\in\mathcal{D}(J)}\langle b,h_I\rangle^2\langle w^{-1}\rangle^2_I\langle w\rangle _I\bigg)^{1/2}\bigg(\frac{1}{|J|}\sum_{I\in\mathcal{D}(J)}|I_-||\Delta_{I_-}w|^2\frac{1}{\langle w\rangle_I^3}\bigg)^{1/2}\nonumber\\
&~~~~~~~~~~~~\leq C[w]^{1/2}_{A_2}\bigg(\frac{1}{|J|}\sum_{I\in\mathcal{D}(J)}\langle b,h_I\rangle^2\langle w^{-1}\rangle\bigg)^{1/2}\bigg(\frac{1}{|J|}\sum_{I\in\mathcal{D}(J)}|I_-||\Delta_{I_-}w|^2\frac{1}{\langle w\rangle_{I_-}^3}\bigg)^{1/2}\nonumber\\
&~~~~~~~~~~~~\leq C[w]_{A^d_2}\|b\|_{BMO^{\,d}}\langle w^{-1}\rangle_J\,.\label{SP:e9}
\end{align}
Inequality (\ref{SP:e9}) due to Lemma \ref{icl1} and
(\ref{GB:e5})\,. Also, the other sum can be estimated by exactly the same method. Thus we have the embedding condition (\ref{SP:e6}). To
see the embedding condition (\ref{SP:e7}), it is enough to show
$$\frac{1}{|J|}\sum_{I\in\mathcal{D}(J)}|\langle b,h_I\rangle\,\langle h_{I_-},w\rangle|\leq C[w]_{A^d_2}\|b\|_{BMO^{\,d}}\langle w\rangle
_J\,,$$ as we did above. We use Cauchy-Schwarz inequality for
embedding condition (\ref{SP:e7}), then
\begin{align}
\frac{1}{|J|}\sum_{I\in\mathcal{D}(J)}|\langle b,h_I\rangle\langle h_{I_-},w\rangle|&=\frac{1}{|J|}\sum_{I\in\mathcal{D}(J)}|\langle b,h_I\rangle|\sqrt{|I_-|}|\Delta_{I_-}w|\nonumber\\
&\leq\bigg(\frac{1}{|J|}\sum_{I\in\mathcal{D}(J)}\langle b,h_I\rangle^2\frac{1}{\langle w^{-1}\rangle_I}\bigg)^{1/2}\bigg(\frac{1}{|J|}\sum_{I\in\mathcal{D}(J)}|I_-||\Delta_{I_-}w|^2\langle w^{-1}\rangle_I\bigg)^{1/2}\nonumber\\
&\leq C\|b\|_{BMO^{\,d}}\langle w\rangle_J^{1/2}\bigg(\frac{[w]_{A_2}}{|J|}\sum_{I\in\mathcal{D}(J)}|I_-||\Delta_{I_-}w|^2\frac{1}{\langle w\rangle_{I_-}}\bigg)^{1/2}\label{SP:e10}\\
&\leq C[w]_{A^d_2}\|b\|_{BMO^{\,d}}\langle w\rangle _J\,.\label{SP:e11}
\end{align}
Here inequality (\ref{SP:e10}) uses Lemma \ref{icl2}, and inequality (\ref{SP:e11}) uses the fact that $\langle w\rangle^{-1}_I\leq 2\langle w\rangle^{-1}_{I_-}$ and Theorem \ref{WSB} after shifting the indices.\\

If we show the embedding condition (\ref{SP:e8}), then we can
immediately finish the estimate for (\ref{SP:e5}) with bound
$C[w]_{A^d_2}\|b\|_{BMO^{\,d}}\|f\|_{L^2(w^{-1})}\|g\|_{L^2(w)}\,.$
Combining this and (\ref{SP:e4}) will give us our desire result.
\end{proof}

\section{Proof for embedding condition (\ref{SP:e8})}

The following lemma lies at the heart of the matter for the proof of the embedding condition (\ref{SP:e8})\,.
\begin{lemma}\label{wls}
There is a positive constant $C$ so that for all dyadic interval $J\in\mathcal{D}$
\begin{equation}
\frac{1}{|J|}\sum_{I\in\mathcal{D}(J)}|I|\langle
w\rangle_I^{1/4}\langle w^{-1}\rangle_I^{1/4}\bigg(\frac{|\Delta_{I_+}w|+|\Delta_{I_-}w|}{\langle w\rangle_I}\bigg)^2\leq C\langle
w\rangle_J^{1/4}\langle
w^{-1}\rangle_J^{1/4}\,,\label{IC:e1}\end{equation} whenever $w$ is a weight. Moreover, if $w\in A^d_2$ then for all $J\in\mathcal{D}$
$$
\frac{1}{|J|}\sum_{I\in\mathcal{D}(J)}|I|\langle
w\rangle_I\langle w^{-1}\rangle_I\bigg(\frac{|\Delta_{I_+}w|+|\Delta_{I_-}w|}{\langle w\rangle_I}\bigg)^2\leq C\,[w]_{A^d_2}\,.$$
\end{lemma}

\begin{proof}[Proof of condition (\ref{SP:e8})]
By using Cauchy-Schwarz inequality and Lemma \ref{wls}, we have:
\begin{align*}
&\frac{1}{|J|}\sum_{I\in\mathcal{D}(J)}\langle b,h_I\rangle\langle
w^{-1}\rangle_I(|\langle h_{I_-},w\rangle|+|\langle
h_{I_+},w\rangle|)\\
&~~~~~~=\frac{1}{|J|}\sum_{I\in\mathcal{D}(J)}\langle
b,h_I\rangle\langle
w^{-1}\rangle_I\sqrt{\frac{|I|}{2}}\big(|\langle
w\rangle_{I_{-+}}-\langle w\rangle_{I_{--}}|+|\langle
w\rangle_{I_{++}}-\langle w\rangle_{I_{+-}}|\,\big)\\
&~~~~~~\leq\frac{1}{\sqrt{2}}\bigg(\frac{1}{|J|}\sum_{I\in\mathcal{D}(J)}\langle
b,h_I\rangle^2\langle w^{-1}\rangle_I\langle
w\rangle_I\bigg)^{1/2}\bigg(\frac{1}{|J|}\sum_{I\in\mathcal{D}(J)}|I|\langle
w^{-1}\rangle_I\langle
w\rangle_I^{-1}(|\Delta_{I_+}w|+|\Delta_{I_-}w|)^2\bigg)^{1/2}\\
&~~~~~~\leq
C[w]_{A^d_2}\|b\|_{BMO^{\,d}}\,.
\end{align*}
\end{proof}
We turn to the proof of Lemma \ref{wls}. In the first place, we need to revisit some properties of function $B(u,v):=\sqrt[4]{uv}$ on the domain $\mathfrak{D}_0$ which is given by
$$\{(u,v)\in\mathbb{R}^2_+\,:~uv\geq 1/2\,\}\,.$$
It is known, we refer to \cite{OB},
that $B(u,v)$ satisfies the following differential inequality in $\mathfrak{D}_0$
\begin{equation}
-(du,dv)d^2B(u,v)(du,dv)^t\geq\frac{1}{8}\frac{v^{1/4}}{u^{7/4}}|du|^2\,.\label{IC:e2}
\end{equation}
Furthermore, this implies the following convexity condition.
For all $(u,v),\,(u_{\pm},v_{\pm})\in\mathfrak{D}_0\,,$
\begin{equation}
B(u,v)-\frac{B(u_+,v_+)+B(u_-,v_-)}{2}\geq C_1\frac{v^{1/4}}{u^{7/4}}(u_+-u_-)^2\,,
\end{equation}
where $u=(u_++u_-)/2$ and $v=(v_++v_-)/2\,.$
Let us define
$$A(u,v,\Delta u):=aB(u,v)+B(u+\Delta u,v)+B(u-\Delta u,v)\,,$$
on the domain $\mathfrak{D}_1$ with some positive constant $a>0\,.$ Here $(u,v,\Delta u)\in\mathfrak{D}_1$ means all pairs $(u,v),\,(u+\Delta u,v),\,(u-\Delta u,v)\in \mathfrak{D}_0\,.$
Then $A$ has the size property and the convexity property:
\begin{equation} \textrm{if}~~(u,v,\Delta u)\in\mathfrak{D}_1,~~\textrm{then}~~0\leq A(u,v,\Delta u)\leq (a+2)\sqrt[4]{uv}\,,\label{IC:e4}
\end{equation} and
\begin{equation}
A(u,v,\Delta u)-\frac{1}{2}\big[A(u_+,v_+,\Delta u_1)+A(u_-,v_-,\Delta u_2)\big]\geq C_2\frac{v^{1/4}}{u^{7/4}}(\Delta u_1^2+\Delta u_2^2)\,,\label{IC:e5}
\end{equation}
where $u=(u_++u_-)/2\,,v=(v_++v_-)/2,\,$ and $\Delta u=(u_+-u_+)/2\,.$
The property (\ref{IC:e5}) is directly from the definition of function $B(u,v)\,.$ At the end, $\Delta u$ will play the role of $\Delta_Iw$, $\Delta u_1$ is $\Delta_{I_+}w$, and $\Delta u_2$ is $\Delta_{I_-}w\,.$
We can rewrite left hand side of the inequality (\ref{IC:e5}) as follows
\begin{align}
&A(u,v,\Delta u)-\frac{1}{2}\big[A(u_+,v_+,\Delta u_1)+A(u_-,v_-,\Delta u_2)\big]\nonumber\\
&=\,aB(u,v)+B(u+\Delta u,v)+B(u-\Delta u,v)-\frac{1}{2}\big[aB(u_+,v_+)+B(u_++\Delta u_1,v_+)+B(u_+-\Delta u_1,v_+)\nonumber\\
&~~~~+aB(u_-,v_-)+B(u_-+\Delta u_1,\Delta u_-)+B(u_--\Delta u_1,v_-)\big]\nonumber\\
&=\,aB(u,v)-\frac{a}{2}(B(u_+,v_+)+B(u_-,v_-))+B(u_+,v)+B(u_-,v)-\frac{1}{2}\Big[B(u+\Delta u+\Delta u_1,v+\Delta v)\nonumber\\
&~~~~+B(u+\Delta u-\Delta u_1,v+\Delta v)+B(u-\Delta u+\Delta u_2,v-\Delta v)+B(u-\Delta u-\Delta u_2,v-\Delta v)\Big]\,.\label{IC:e6}
\end{align}
Using Taylor's theorem:
\begin{equation*}
B(u+u_0,v+v_0)=B(u,v)+\nabla B(u,v)(u_0,v_0)^t+\int^1_0(1-s)(u_0,v_0)d^2B(u+su_0,v+sv_0)(u_0,v_0)^tds\,,\end{equation*}
and the convexity condition of $B(u,v)$, we are going to estimate the lower bounds of (\ref{IC:e6}).
\begin{align}
&-\frac{1}{2}B(u+\Delta u+\Delta u_1,v+\Delta v)\nonumber\\
&=-\frac{1}{2}\Big(B(u,v)+\nabla B(u,v)(\Delta u+\Delta u_1,\Delta v)^t\Big)\nonumber\\
&~~~~-\int^1_0(1-s)(\Delta u+\Delta u_1,\Delta v)d^2B(u+s(\Delta u+\Delta u_1),v+s\Delta v)(\Delta u+\Delta u_1,\Delta v)^tds\nonumber\\
&\geq -\frac{1}{2}\Big(B(u,v)+\nabla B(u,v)(\Delta u+\Delta u_1,\Delta v)^t\Big)+\frac{1}{8}\int^1_0(1-s)\frac{(v+s\Delta v)^{1/4}}{(u+s(\Delta u+\Delta u_1))^{7/4}}(\Delta u+\Delta u_1)^2ds\nonumber\\
&\geq -\frac{1}{2}\Big(B(u,v)+\nabla B(u,v)(\Delta u+\Delta u_1,\Delta v)^t\Big)+\frac{(\Delta u+\Delta u_1)^2}{8(4u)^{7/4}}\int^1_0(1-s)(v+s\Delta v)^{1/4}ds\label{IC:e7}\\
&\geq -\frac{1}{2}\Big(B(u,v)+\nabla B(u,v)(\Delta u+\Delta u_1,\Delta v)^t\Big)+\frac{(\Delta u+\Delta u_1)^2v^{1/4}}{8(4u)^{7/4}}\int^1_0(1-s)(1+s\frac{\Delta v}{v})^{1/4}ds\nonumber\\
&\geq -\frac{1}{2}\Big(B(u,v)+\nabla B(u,v)(\Delta u+\Delta u_1,\Delta v)^t\Big)+\frac{1}{72\cdot4^{3/4}}\frac{v^{1/4}}{u^{7/4}}(\Delta u+\Delta u_1)^2\,.\label{IC:e8}
\end{align}
Inequality (\ref{IC:e7}) is due to the following inequalities $$|\Delta u|=\frac{|u_+-u_-|}{2}\leq \frac{|u_++u_-|}{2}= u~\textrm{and}~|\Delta u_1|=\frac{|u_{++}-u_{+-}|}{2}\leq u_+\leq 2u\,.$$
Since $(1-s)^{1/4}\leq(1-|\beta|s)^{1/4}\leq (1+\beta s)^{1/4}\,$ for any $|\,\beta|<1$, it is clear that
$$\int^1_0(1-s)(1+\beta s)^{1/4}ds\geq\int^1_0(1-s)^{5/4}ds=\frac{4}{9}\,,$$ and this allows the inequality (\ref{IC:e8}). With the same arguments, we also estimate the following lower
bounds:
\begin{align}
&-\frac{1}{2}\Big[B(u+\Delta u-\Delta u_1,v+\Delta v)+B(u-\Delta u+\Delta u_2,v-\Delta v)+B(u-\Delta u-\Delta u_2,v-\Delta v)\Big]\nonumber\\
&~~~~\geq -\frac{1}{2}\Big(B(u,v)+\nabla B(u,v)(\Delta u-\Delta u_1,\Delta v)^t\Big)+\frac{1}{72\cdot4^{3/4}}\frac{v^{1/4}}{u^{7/4}}(\Delta u-\Delta u_1)^2\nonumber\\
&~~~~~~~-\frac{1}{2}\Big(B(u,v)+\nabla B(u,v)(-\Delta u+\Delta u_2,-\Delta v)^t\Big)+\frac{1}{72\cdot4^{3/4}}\frac{v^{1/4}}{u^{7/4}}(-\Delta u+\Delta u_2)^2\nonumber\\
&~~~~~~~-\frac{1}{2}\Big(B(u,v)+\nabla B(u,v)(-\Delta u-\Delta u_2,-\Delta v)^t\Big)+\frac{1}{72\cdot4^{3/4}}\frac{v^{1/4}}{u^{7/4}}(\Delta u+\Delta u_2)^2\,.\label{IC:e9}
\end{align}
We can have the following inequality by combining (\ref{IC:e8}), (\ref{IC:e9}) and (\ref{IC:e6}),
\begin{align}
&A(u,v,\Delta u)-\frac{1}{2}\big[A(u_+,v_+,\Delta u_1)+A(u_-,v_-,\Delta u_2)\big]\nonumber\\
&~~~~~~~~~~~~\geq\,(a-2)B(u,v)-\frac{a}{2}(B(u_+,v_+)+B(u_-,v_-))+B(u_+,v)+B(u_-,v)\nonumber\\
&~~~~~~~~~~~~~~~~+\frac{1}{36\cdot4^{3/4}}\frac{v^{1/4}}{u^{7/4}}(2\Delta u^2+\Delta u_1^2+\Delta u_2^2)\nonumber\\
&~~~~~~~~~~~~\geq\,\frac{1}{36\cdot4^{3/4}}\frac{v^{1/4}}{u^{7/4}}(\Delta u_1^2+\Delta u_2^2)\,.\label{IC:e10}
\end{align}
To see the inequality (\ref{IC:e10}), using convexity condition of $B(u,v)=\sqrt[4]{uv}$ and inequality: $(1-s)u\leq u-s\Delta u\leq u+s\Delta u\,,$
\begin{align}
&(a-2)B(u,v)-\frac{a}{2}\Big(B(u_+,v_+)+B(u_-,v_-)\Big)+B(u_+,v)+B(u_-,v)\nonumber\\
&~~~~~~~~~~~~= a\bigg(B(u,v)-\frac{1}{2}(B(u_+,v_+)+B(u_-,v_-)\bigg)\nonumber\\
&~~~~~~~~~~~~~~~-\bigg(\frac{3}{16}\Delta u^2\int^1_0(1-s)v^{1/4}(u+s\Delta u)^{-7/4}ds+\frac{3}{16}\Delta u^2\int^1_0(1-s)v^{1/4}(u-s\Delta u)^{-7/4}ds\bigg)\nonumber\\
&~~~~~~~~~~~~\geq aC_1\frac{v^{1/4}}{u^{7/4}}\Delta u^2-\frac{6}{16}\Delta u^2\frac{v^{1/4}}{u^{7/4}}\int^1_0(1-s)^{-3/4}ds\nonumber\\
&~~~~~~~~~~~~=  aC_1\frac{v^{1/4}}{u^{7/4}}\Delta u^2-\frac{3}{2}\Delta u^2\frac{v^{1/4}}{u^{7/4}}\nonumber\\
&~~~~~~~~~~~~=\bigg(aC_1-\frac{3}{2}\bigg)\frac{v^{1/4}}{u^{7/4}}\Delta u^2\,.\label{IC:p1}
\end{align}
Choosing a constant $a$ sufficiently large so that $aC_1>3/2$, quantity in (\ref{IC:p1}) remains positive.
This observation and discarding nonnegative terms yield inequality (\ref{IC:e10}). Choosing the constant $C_2=1/(36\cdot4^{3/4})$ in (\ref{IC:e5}) completes the proof of the concavity property of $A(u,v,\Delta u)\,.$ We now turn to the proof of Lemma \ref{wls}.
\begin{proof}[Proof of Lemma \ref{wls}]
Let $u_I:=\langle
w\rangle_I,\,v_I:=\langle
w^{-1}\rangle_I,\,u_{\pm}=u_{I_{\pm}},\,v_{\pm}=v_{I_{\pm}},\,\Delta u_I=\Delta_Iw,\,\Delta u_1=\Delta u_{I_+},$ and $\Delta u_2=\Delta u_{I_-}\,.$ Then by H\"{o}lder's inequality $(u,v,\Delta u),\, (u_+,v_+,\Delta u_1),\,$ and $(u_-,v_-,\Delta u_2)$ belong to $\mathfrak{D}_1\,.$ Fix $J\in\mathcal{D}\,,$ by properties (\ref{IC:e4}) and (\ref{IC:e5})
\begin{align*}
&(a+2)|J|\sqrt[4]{\langle w\rangle_J\langle w^{-1}\rangle_J}\geq
|J|A(u_J,v_J,\Delta u_J)\\
&~~~~~~~~~~~~\geq\,\frac{1}{2}\bigg(|J_+|A(u_+,v_+,\Delta u_1)+|J_-|A(u_-,v_-,\Delta u_2)\bigg)+|J|C\frac{\langle w^{-1}\rangle_J}{\langle
w\rangle_J^{7/4}}(|\Delta_{J_+}u|^2+|\Delta_{J_-}u|^2)\,.
\end{align*}
Since $A(u,v,\Delta u)\geq 0$, iterating the above process will yield
\begin{equation}
|J|\sqrt[4]{\langle w\rangle_J\langle w^{-1}\rangle_J}\geq
C\sum_{I\in\mathcal{D}(J)}|I|\langle w^{-1}\rangle_I^{1/4}\langle
w\rangle^{-7/4}_I(|\Delta_{I_+}w|^2+|\Delta_{I_-}w|^2)\,.\label{IC:e12}
\end{equation}
Also, one can easily have
\begin{equation}
|J|\sqrt[4]{\langle w\rangle_J\langle w^{-1}\rangle_J}\geq
C\sum_{I\in\mathcal{D}(J)}|I|\langle w^{-1}\rangle_I^{1/4}\langle
w\rangle^{-7/4}_I\,\Delta_{I_+}w^2\,,\label{IC:e13}\end{equation}
and
\begin{equation}|J|\sqrt[4]{\langle w\rangle_J\langle w^{-1}\rangle_J}\geq
C\sum_{I\in\mathcal{D}(J)}|I|\langle w^{-1}\rangle_I^{1/4}\langle
w\rangle^{-7/4}_I\,\Delta_{I_-}w^2\,.\label{IC:e14}
\end{equation}
Then,
\begin{align*}
&\frac{1}{|J|}\sum_{I\in\mathcal{D}(J)}|I|\langle w^{-1}\rangle_I^{1/4}\langle
w\rangle^{-7/4}_I(|\Delta_{I_+}w|+|\Delta_{I_-}w|)^2\\
&~~~~~~~~~~~~=\frac{1}{|J|}\bigg(\sum_{I\in\mathcal{D}(J)}|I|\langle w^{-1}\rangle_I^{1/4}\langle
w\rangle^{-7/4}_I(|\Delta_{I_+}w|^2\\
&~~~~~~~~~~~~~~~~+|\Delta_{I_-}w|^2)+2\sum_{I\in\mathcal{D}(J)}|I|\langle w^{-1}\rangle_I^{1/4}\langle
w\rangle^{-7/4}_I(|\Delta_{I_+}w|\,|\Delta_{I_-}w|)\bigg)\\
&~~~~~~~~~~~~\leq\frac{1}{|J|}\Bigg(\sum_{I\in\mathcal{D}(J)}|I|\langle w^{-1}\rangle_I^{1/4}\langle
w\rangle^{-7/4}_I(|\Delta_{I_+}w|^2+|\Delta_{I_-}w|^2)\\
&~~~~~~~~~~~~~~~~+2\bigg(\sum_{I\in\mathcal{D}(J)}|I|\langle w^{-1}\rangle_I^{1/4}\langle
w\rangle^{-7/4}_I\,\Delta_{I_+}w^2\bigg)^{1/2}\bigg(\sum_{I\in\mathcal{D}(J)}|I|\langle w^{-1}\rangle_I^{1/4}\langle
w\rangle^{-7/4}_I\,\Delta_{I_-}w^2\bigg)^{1/2}\Bigg)\\
&~~~~~~~~~~~~\leq \frac{3}{C}\sqrt[4]{\langle w\rangle_I\langle w^{-1}\rangle_I}\,.
\end{align*}
\end{proof}
\section{Recently developed tools and their applications}
The dyadic shift operator was first introduced in \cite{P1} to
replace the weighted norm estimate for the Hilbert transform. It
was also encountered in \cite{PTV}, so Riesz transforms can be
obtained as the result of averaging some dyadic shift operator.
Recently, in  \cite{LPS} and \cite{CMP}, a more general class of
dyadic shift operators, so called the Haar shift operators were
introduced. The Hilbert transform, Riesz transforms, and
Beurling-Ahlfors operator are in the convex hull of this class, as
they can be written as appropriate averages of Haar shift
operators. Let $\mathcal{D}^n$ denote the collection of dyadic
cubes in $\R^n\,,$ $\mathcal{D}^n(Q)$ denotes dyadic subcubes of
$Q\,,$ and $|Q|$ denotes the volume of the dyadic cube $Q\,.$ We
start with some definitions.
\begin{definition}\label{DHf}
A Haar function on a cube $Q\subset \R^n$ is a function $H_Q$ such that
\begin{enumerate}
  \item $H_Q$ is supported on $Q$, and is constant on $\mathcal{D}^n(Q)\,.$
  \item $\|H_Q\|_{\infty}\leq |Q|^{-1/2}\,.$
  \item $H_Q$ has a mean zero.
\end{enumerate}
\end{definition}
\begin{definition}\label{DLaCa}
Given an integer $\tau>0\,,$ we say an operator of the following
form is in \emph{the first class of Haar shift operators of index
$\tau\,$}
$$T_{\tau}f(x)=\sum_{Q\in\mathcal{D}^n}\sum_{\substack{Q',Q''\in\mathcal{D}(Q)\\ 2^{-\tau n}|Q|\leq |Q'|,|Q''|}}a_{Q',Q''}\langle f,H_{Q'}\rangle H_{Q''}(x)\,,$$
where the constant $a_{Q',Q''}$ satisfy the following size condition: \begin{equation}|\,a_{Q',Q''}|\leq C\bigg(\frac{|Q'|}{|Q|}\cdot\frac{|Q''|}{|Q|}\bigg)^{1/2}\,.\label{RA:e1}\end{equation}
\end{definition}
Note that once a choice of Haar functions has been made
$\{H_Q\}_{Q\in\mathcal{D}^n}\,,$ then this is an orthogonal
family, such that $\|H_{Q}\|_{L^2}\leq 1\,,$ so one could
normalize in $L^2\,.$ Note that one can easily see that the dyadic
shift operator $S$ belongs to the first class of a Haar shift
operator of index $\tau=1\,$ with
$$a_{I',I''}=\left\{\begin{array}{l}
\pm 1\quad\textrm{for }I'=I,~~I''=I_{\mp}\\
\hspace{0.2cm}0\quad\textrm{ otherwise }\,.
\end{array}\right.$$
One of the main result in \cite{LPS} and \cite{CMP} is the following
\begin{theorem}[\cite{LPS}, \cite{CMP}]\label{LaCa} Let $T$ be in the first class of Haar shift operators of index $\tau\,.$ Then
for all $w\in A^d_2\,,$ there exists $C(\tau, n)$ which only
depends on $\tau$ and $n$ such that $$\|T\|_{L^2(w)\ra L^2(w)}\leq
C(\tau,n)[w]_{A^d_2}\,.$$
\end{theorem}
As a consequence of this Theorem, linear bounds for the Hilbert
transform, Riesz transforms, and the Beurling-Ahlfors operator are
recovered. There are now two different proofs of Theorem
(\ref{LaCa}) in \cite{LPS} and \cite{CMP}. The commutator
$[\lambda_b, S]\,$ is also in the first class of Haar shift
operators of index $\tau=1\,.$ Recall the observation in the
section 4,
$$[\lambda_b,S](f)=-\sum_{I\in\mathcal{D}}\Delta_Ib\langle f,h_I\rangle(h_{I_+}+h_{I_-})\,.$$ Then we can see
$$a_{I',I''}=\left\{\begin{array}{l}
-\Delta_Ib\quad\textrm{for }I'=I,~~I''=I_{\pm}\\
\hspace{0.2cm}0\quad\textrm{ otherwise }\,,
\end{array}\right.$$
moreover $|\,a_{I',I''}|=|\Delta_I b|\leq 2\|b\|_{BMO^{\,d}}\,$
this means the constant $a_{I',I''}$ satisfy the size condition
(\ref{RA:e1}) with $C=2\sqrt{2}\|b\|_{BMO^{\,d}}\,.$ These
observations, Theorem \ref{lbp}, and Theorem \ref{LaCa}
immediately recover the quadratic bound for the commutator of the
Hilbert transform which was proved in this paper. We now define
the second class of Haar shift operators of index $\tau\,.$
\begin{definition}\label{D2LaCa}
Given an integer $\tau>0\,,$ we say an operator $T$ of the form in
Definition \ref{DLaCa} is in \emph{the second class of Haar shift
operators of index $\tau\,,$} if $T$ is bounded on $L^2$ and the
function $H_Q$ satisfy the condition (a) and (b) in Definition
\ref{DLaCa}.
\end{definition}
The second class of Haar shift operators is more general than the
first class. One can easily observe that the operators $\pi_b$,
$S\pi_b$ and $\pi^{\ast}_bS$ do not satisfy the condition (c) on
Definition \ref{DHf}, however these operators satisfy the
conditions of Definition \ref{D2LaCa}. Note that the $n$-variable
paraproduct is a sum of $2^n-1$ of the second class of Haar shift
operator of index 1, the restricted $n$-variable dyadic
paraproduct $$\pi_bf=\sum_{Q\in\mathcal{D}^n}\langle
f\rangle_Q\langle b,H_Q\rangle H_Q\,.$$ In \cite{HLRV}, the linear
estimate for the maximal truncations of these operators is
presented. This also recovers our linear bound estimates for
$S\pi_b$ and $\pi^{\ast}_bS$. On the other hand, authors in
\cite{CMP} also reproduce the linear estimate for the dyadic
paraproduct with different technique.
\begin{lemma}\label{lExCo}
Let $T_{\tau}$ a Haar shift operator of the first class, then
$[\lambda_b,T_{\tau}]$ is an operator of the same class.
\end{lemma}
\begin{proof}
We are going to use the restricted multi-variable $\lambda_b$
operator which is
$$\lambda_bf=\sum_{Q\in\mathcal{D}^n}\langle b\rangle_Q\langle
f,H_Q\rangle H_Q\,.$$ One can get the $n$-variable $\lambda_b$
operator by summing over $2^n-1$ of restricted $\lambda_b$
operator. Observe that,
\begin{align*}
[\lambda_b,T_{\tau}]f&=\lambda_b(T_{\tau}f)-T_{\tau}(\lambda_b f)\\
&=\sum_{Q\in\mathcal{D}_n}\sum_{\substack{Q',Q''\in\mathcal{D}(Q)\\ 2^{-\tau n}|Q|\leq |Q'|,|Q''|}}a_{Q',Q''}\langle b\rangle_{Q''}\langle f,H_{Q'}\rangle H_{Q''}-\sum_{Q\in\mathcal{D}_n}\sum_{\substack{Q',Q''\in\mathcal{D}(Q)\\ 2^{-\tau n}|Q|\leq |Q'|,|Q''|}}a_{Q',Q''}\langle b\rangle_{Q'}\langle f,H_{Q'}\rangle H_{Q''}\\
&=\sum_{Q\in\mathcal{D}_n}\sum_{\substack{Q',Q''\in\mathcal{D}(Q)\\ 2^{-\tau n}|Q|\leq |Q'|,|Q''|}}a_{Q',Q''}(\langle b\rangle_{Q''}-\langle b\rangle_{Q'})\langle f,H_{Q'}\rangle H_{Q''}\,.
\end{align*}
Since $$\big|\,a_{Q',Q''}(\langle b\rangle_{Q''}-\langle
b\rangle_{Q'})\big|\leq \|b\|_{BMO}|\,a_{Q',Q''}|\,,$$
$[\lambda_b,T_{\tau}]$ remains in the same class of
$T_{\tau}\,.$\end{proof} Theorem \ref{LaCa} and Lemma \ref{lExCo}
allow to extend our result to more general class of commutators
including the Riesz transforms and the Beurling-Ahlfors operator
as in Theorem \ref{ExCo}.

\section{Sharp bounds}
In this section, we start proving that the quadratic estimate in
Theorem \ref{Mr:t1} is sharp, by showing an example which returns
quadratic bound. This example was discovered by C. Per\'{e}z
\cite{PP} who is kindly allowing us to reproduce it in this paper.
The same calculations show that the bounds in Theorem \ref{Pr:Mr}
are also sharp for $p\neq 2$ and $1<p<\infty\,.$ Variations over
this example will then show that the bounds in Theorem \ref{ExCo}
are sharp for the Riesz transforms and the Beurling-Ahlfors
operator as well.
\subsection{The Hilbert transform}
Consider the weight, for $0<\delta<1\,$:
$$w(x)=|x|^{1-\delta}\,.$$
It is well known that $w$ is an $A_2$ weight and
$$[w]_{A^2}\sim\frac{1}{\delta}\,.$$ We now consider the function
$f(x)=x^{-1+\delta}\chi_{(0,1)}(x)$ and BMO function
$b(x)=\log|x|\,.$ We claim that $$\|[b,H]f(x)|\geq
\frac{1}{\delta^2}f(x)\,.$$ For $0<x<1\,,$ we have
\begin{align*}
[b,H]f(x)&=\int_0^1\frac{\log x-\log y}{x-y}y^{-1+\delta}dy=\int_0^1\frac{\log(x/y)}{x-y}y^{-1+\delta}dy\\
&=x^{-1+\delta}\int_0^{1/x}\frac{\log\Big(\frac{1}{t}\Big)}{1-t}t^{-1+\delta}dt\,.
\end{align*}
Now,
$$\int_0^{1/x}\frac{\log(1/t)}{1-t}t^{-1+\delta}dt=\int_0^1\frac{\log(1/t)}{1-t}t^{-1+\delta}dt+\int_1^{1/x}\frac{\log(1/t)}{1-t}t^{-1+\delta}dt\,,$$
and since $\frac{\log(1/t)}{1-t}$ is positive for
$(0,1)\cup(1,\infty)$ we have for $0<x<1$
\begin{equation}|[b,H]f(x)|>x^{-1+\delta}\int_0^1\frac{\log(1/t)}{1-t}t^{-1+\delta}dt\,.\label{SE:e1}\end{equation}
But since
\begin{equation}\int_0^1\frac{\log(1/t)}{1-t}t^{-1+\delta}dt>\int_0^1\log(1/t)t^{-1+\delta}dt=\int_0^{\infty}se^{-s\delta}ds=\frac{1}{\delta^2}\,,\label{SE:e2}\end{equation}
our claim follows and
$$\|[b,H]f\|_{L^2(w)}\geq \frac{1}{\delta^2}\|f\|_{L^2(w)}\sim[w]^2_{A_2}\|f\|_{L^2(w)}\,.$$
A first approximation of what the bounds in $L^p(w)$ is given by
an application of the sharp extrapolation theorem for the upper
bound, paired with the knowledge of the sharp bound on $L^2(w)$ to
obtain a lower bound.
\begin{proposition}\label{Mr:c2}
For $1< p<\infty$ there exist constants $c$ and $C$ only depending
on $p$ such that
\begin{equation}c[w]^{2\min\{1,\frac{1}{p-1}\}}_{A_p}\|b\|_{BMO}\leq\|[b,H]\|_{L^p(w)\ra
L^p(w)}\leq
C[w]^{2\max\{1,\frac{1}{p-1}\}}_{A_p}\|b\|_{BMO}\,,\label{Mre1}\end{equation}for
all $b\in BMO\,.$
\end{proposition}
\begin{proof} Because the upper bound in (\ref{Mre1}) is the direct
consequence of the quadratic bound in the Theorem \ref{Mr:t1} and
sharp extrapolation theorem, we will only prove the lower bound.
Let us assume that, for $1<r<2\,$ and $\alpha<1\,,$
$$\|[b,H]\|_{L^r(w)\ra L^r(w)}\leq C[w]^{2\alpha}_{A_r}\|b\|_{BMO}\,.$$
This and the sharp extrapolation theorem return
$$\|[b,H]\|_{L^2(w)\ra L^2(w)}\leq C[w]^{2\alpha}_{A_2}\|b\|_{BMO}\,.$$
This contradicts to the sharpness $(p=2)$\,. Similarly, one can
conclude for $p>2\,.$
\end{proof}
We now consider
the weight $w(x)=|x|^{(1-\delta)(p-1)}$ then $w$ is an $A_p$
weight with $[w]_{A_p}\sim\delta^{1-p}\,.$ By (\ref{SE:e1}) and
(\ref{SE:e2}) we have
$$\|[b,H]f\|_{L^p(w)}\geq \frac{1}{\delta^2}\|f\|_{L^p(w)}=(\delta^{1-p})^{\frac{2}{p-1}}\|f\|_{L^p(w)}\sim [w]_{A_p}^{\frac{2}{p-1}}\|f\|_{L^p(w)}\,.$$
This shows the upper bound in (\ref{Mre1}) is sharp for $1<p\leq
2\,.$ We use the duality argument to see the sharpness of the
quadratic estimate for $p>2\,.$ Note that the commutator is a
self-adjoint operator:
$$\langle bH(f)-H(bf),g\rangle=\langle
f,H^{\ast}(bg)\rangle-\langle f,bH^{\ast}(g)\rangle=\langle
f,bH(g)-H(bg)\rangle\,.$$ Consider $1<p\leq 2$ and set
$u=w^{1-p'}\,,$ then
\begin{align}
\|[b,H]\|_{L^{p'}(u)\ra
L^{p'}(u)}&=\|[b,H]\|_{L^{p'}(w^{1-p'})\ra L^{p'}(w^{1-p'})}=\|[b,H]^{\ast}\|_{L^{p'}(w^{1-p'})\ra L^{p'}(w^{1-p'})}\nonumber\\
&=\|[b,H]\|_{L^p(w)\ra L^p(w)}\leq
C\|b\|_{BMO^d}[w]^{\frac{2}{1-p}}_{A_p}\label{SE:e3}\\
&=C\|b\|_{BMO}[w^{1-p'}]^2_{A_{p'}}=C\|b\|_{BMO}[u]^2_{A_{p'}}\,.\nonumber
\end{align}
Since the inequality in (\ref{SE:e3}) is sharp, we can conclude
that the result of Theorem \ref{Pr:Mr} is also sharp for $p>2\,.$
\subsection{Beurling-Ahlfors operator}
Recall the Beurling-Ahlfors operator $\mathcal{B}$ is given by
convolution with the distributional kernel  $p.v.1/z^2$:
$$\mathcal{B}f(x,y)=p.v.\frac{1}{\pi}\int_{\R^2}\frac{f(x-u,y-v)}{(u+iv)^2}\,dudv\,.$$
Then the commutator of the Beurling-Ahlfors operator can be
written:
$$[b,\mathcal{B}]f(x,y)=p.v.\frac{1}{\pi}\int_{\R^2}\frac{b(x,y)-b(s,t)}{((x-s)+i(y-t))^2}\,f(s,t)\,dsdt\,.$$
It was observed, in \cite{DV}, that the linear bound for the
Beurling-Ahlfors operator is sharp in $L^2(w)\,,$ with weights
$w(z)=|\,z|^{\alpha}$ and functions $f(z)=|\,z|^{-\alpha}\,$ where
$|\,\alpha|<2\,.$ Similarly, we consider weights
$w(z)=|\,z|^{2-\delta}$ where $0<\delta<1$. Note that
$w(z)=|\,z|^{2-\delta}:\C\ra[0,\infty)$ is a $A_2$-weight with
$[w]_{A_2}\sim \delta^{-1}\,.$ We also consider a BMO function
$b(x)=\log|\,z|\,.$ Let
$E=\{(r,\theta)\,|\,0<r<1,\,0<\theta<\pi/2\}$ and
$\Omega=\{(r,\theta)\,|\,1<r<\infty,\,\pi<\theta<3\pi/2\}$ We are
going to estimate $|[b,\mathcal{B}]f(z)|$ for $z\in \Omega$ with a
function $f(z)=|\,z|^{\delta-2}\chi_{E}(z)\,.$ Let $z=x+iy$ and
$\zeta=s+ti\,.$ Then, for $z\in \Omega\,,$
\begin{align*}|[b,\mathcal{B}]f(z)|&=\frac{1}{\pi}\bigg|\int_{E}\frac{(b(z)-b(\zeta))f(\zeta)}{(z-\zeta)^2}\,d\zeta\bigg|\\
&=\frac{1}{\pi}\bigg|\int_{E}\frac{b(x,y)-b(s,t)}{((x-s)+i(y-t))^2}f(s,t)dsdt\bigg|\\
&=\frac{1}{\pi}\bigg|\int_{E}\frac{(\log|\,z|-\log|\,\zeta|)|\,\zeta|^{\delta-2}((x-s)^2-(y-t)^2)}{((x-s)^2+(y-t)^2)^2}\,dsdt\\
&\qquad\quad\qquad+i\int_{E}\frac{(\log|\,z|-\log|\,\zeta|)|\,\zeta|^{\delta-2}(2(x-s)(y-t))}{((x-s)^2+(y-t)^2)^2}\,dsdt\bigg|\,.
\end{align*}
For $z\in \Omega$ and $\zeta\in E\,,$ we have $(x-s)(y-t)\geq
xy\,$ and by triangle inequality
$((x-s)^2+(y-t)^2)^2=|\,z-\zeta|^4\leq (|\,z|+|\,\zeta|)^4\,.$
After neglecting the positive term (real part), we get
\begin{align*}|[b,\mathcal{B}]f(z)|^2&\geq
\frac{4}{\pi^2}\bigg(xy\int_E\frac{\log(|\,z|/|\,\zeta|)|\,\zeta|^{\delta-2}}{(|\,z|+|\,\zeta|)^4}\,dsdt\bigg)^2
=\frac{4}{\pi^2}\bigg(xy\int_0^{\pi/2}\int_0^1\frac{\log(|\,z|/r)r^{\delta-2}r}{(|\,z|+r)^4}\,drd\theta\bigg)^2\\
&=x^2y^2\bigg(\frac{1}{|\,z|^4}\int_0^1\frac{\log(|\,z|/r)r^{\delta-1}}{(1+r/|\,z|)^4}dr\bigg)^2
=x^2y^2\bigg(\frac{1}{|\,z|^4}\int^{1/|z|}_0\frac{\log(1/t)(|\,z|t)^{\delta-1}}{(1+t)^4}\,|\,z|dt\bigg)^2\\
&=x^2y^2\bigg(\frac{1}{|\,z|^{4-\delta}}\int^{1/|z|}_0\frac{\log(1/t)t^{\delta-1}}{(1+t)^4}\,dt\bigg)^2\,.
\end{align*}
Since $|\,z|/(|\,z|+1)\leq 1/(1+t)\,,$ for $t<1/|\,z|\,,$ we have
\begin{align*}
|[b,\mathcal{B}]f(z)|^2&\geq
x^2y^2\bigg(\frac{1}{|\,z|^{-\delta}(|\,z|+1)^4}\int_0^{1/|z|}\log(1/t)t^{\delta-1}\,dt\bigg)^2\\
&=\frac{x^2y^2}{|\,z|^{-2\delta}(|\,z|+1)^8}\bigg(\frac{|\,z|^{-\delta}(1+\delta\log|\,z|)}{\delta^2}\bigg)^2\\
&=\frac{x^2y^2}{(|\,z|+1)^8}\frac{(1+\delta\log|\,z|)^2}{\delta^4}\,.
\end{align*}
Then, we can estimate the $L^2(w)$-norm as follows.
\begin{align*}\|[b,\mathcal{B}]f\|^2_{L^2(w)}&\geq\frac{1}{\delta^4}\int_{\Omega}\frac{x^2y^2(1+\delta\log|\,z|)^2}{(|\,z|+1)^8}\,|\,z|^{2-\delta}\,dxdy\\
&=\frac{1}{\delta^4}\int_1^{\infty}\int_{\pi}^{3\pi/2}\frac{r^4\cos^2\theta\sin^2\theta(1+\delta\log
r)^2}{(r+1)^8}\,r^{3-\delta}drd\theta\\
&=\frac{\pi}{\delta^416}\int_1^{\infty}\frac{r^{7-\delta}(1+\delta\log
r)^2}{(r+1)^8}\,dr\geq\frac{\pi}{\delta^416}\int_1^{\infty}\frac{r^{7-\delta}(1+\delta\log
r)^2}{(2r)^8}\,dr\\
&=\frac{\pi}{\delta^42^{12}}\int_1^{\infty}(1+\delta\log
r)^2r^{-1-\delta}\,dr\\
&=\frac{\pi}{\delta^42^{12}}\bigg(\frac{1}{\delta}+\frac{2\delta}{\delta^2}+\frac{2\delta^2}{\delta^3}\bigg)=\frac{5\pi}{2^{12}}\cdot\frac{1}{\delta^5}\,.
\end{align*}
Combining with $\|f\|^2_{L^2(w)}=\pi/2\delta\,,$ we have that
$\|[b,\mathcal{B}]f\|_{L^2(w)}/\|f\|_{L^2(w)}\sim\delta^{-2}\,,$
which allows to conclude that the quadratic bound for the
commutator with the Beurling-Ahlfors operators is sharp in
$L^2(w)\,.$ Same calculations with weights
$w(z)=|\,z|^{(2-\delta)(p-1)}$ and functions
$f(z)=|\,z|^{(\delta-2)(p-1)}\,$ will provide the sharpness for
$1<p\leq 2\,,$ and it is sufficient to conclude for all
$1<p<\infty$ because the Beurling-Ahlfors operator is essentially
self adjoint operator
($\mathcal{B}^{\ast}=e^{i\phi}\mathcal{B})\,,$ so the commutator
of the Beurling-Ahlfors operator is also self adjoint.
\subsection{Riesz transforms}
Consider weights $w(x)=|\,x|^{n-\delta}$ and functions
$f(x)=x^{\delta-n}\chi_{E}(x)\,$ where $E=\{x\,|\,x \in (0,1)^n
\cap B(0,1)\}\,,$  and a BMO function $b(x)=\log|\,x|\,.$ It was
observed that $|\,x|^{n-\delta}$ is an $A_2$-weight in $\R^n$ with
$[w]_{A^2}\sim\delta^{-1}\,.$ We are going to estimate $[b,R_j]f$
over the set $\Omega=\{y\in B(0,1)^c\,|\,y_i<0\textrm{ for all
}i=1,2,...,n\}\,,$ where $R_j$ stands for the $j$-th direction
Riesz transform on $\R^n$ and is defined as follows:
$$R_jf(x)=c_n p.v.\int_{\R^n}\frac{y_j}{|\,y|^{n+1}}\,f(x-y)\,dy\,,\qquad1\leq j\leq n\,,$$
where $c_n=\Gamma((n+1)/2)/\pi^{(n+1)/2}\,.$ One can observe that,
for all $x\in E$ and fixed $y\in\Omega\,,$
$$|\,y_j-x_j|\geq |\,y_j|\textrm{ and } |\,y-x|\leq |\,y|+|\,x|\,.$$
$C_n$ denotes a constant depending only on the dimension that may
change from line to line. Then,
\begin{align*}
|[b,R_j]f(y)|&=\bigg|\int_{E}\frac{(y_j-x_j)(\log|\,y|-\log|\,x|)|\,x|^{\delta-n}}{|\,y-x|^{n+1}}\,dx\bigg|
\geq |\,y_j|\int_{E}\frac{\log(|\,y|/|\,x|)|\,x|^{\delta-n}}{(|\,y|+|\,x|)^{n+1}}\,dx\\
&\geq |\,y_j|\int_{E\cap S^{n-1}}\int_0^{1}\frac{\log(|\,y|/r)
r^{\delta-n}r^{n-1}}{(|\,y|+r)^{n+1}}\,drd\sigma
=C_n|\,y_j|\int_0^{1/|y|}\frac{\log(1/t)(t|\,y|)^{\delta-1}|\,y|}{(|\,y|+|\,y|t)^{n+1}}\,dt\\
&=\frac{C_n|\,y_j|}{|\,y|^{n+1-\delta}}\int_0^{1/|y|}\frac{\log(1/t)t^{\delta-1}}{(1+t)^{n+1}}\,dt
\geq\frac{C_n|\,y_j|}{|y|^{n+1-\delta}}\bigg(\frac{|\,y|}{|\,y|+1}\bigg)^{n+1}\int_0^{1/|y|}\log(1/t)t^{\delta-1}\,dt\\
&=\frac{C_n|\,y_j|}{|\,y|^{-\delta}(|\,y|+1)^{n+1}}\bigg(\frac{|\,y|^{-\delta}(1+\delta\log|\,y|)}{\delta^2}\bigg)
=\frac{C_n|\,y_j|}{(|\,y|+1)^{n+1}}\frac{1+\delta\log|\,y|}{\delta^2}\,.\end{align*}
We now can bound from below the $L^2(w)$-norm as follows.
\begin{align*}
\|[b,R_j]f(x)\|^2_{L^2(w)}&>\frac{C_n}{\delta^4}\int_{\Omega}\frac{y_j^2(1+\delta\log|\,y|)^2}{(|\,y|+1)^{2n+2}}\,|\,y|^{n-\delta}dy\\
&\geq\frac{C_n}{\delta^4}\int_{\Omega\cap
S^{n-1}}\int_1^{\infty}\frac{\gamma_j^2r^2(1+\delta\log
r)^2r^{n-\delta}r^{n-1}}{(r+1)^{2n+2}}\,drd\sigma(\gamma)\\
&\geq\frac{C_n}{\delta^4}\int_1^{\infty}\frac{(1+\delta\log
r)^2r^{2n-\delta+1}}{r^{2n+2}}\,dr\\
&=\frac{C_n}{\delta^4}\int_1^{\infty}(1+\delta\log r)^2
r^{-\delta-1}\,dr=\frac{C_n}{\delta^5}\,,
\end{align*}
which establishes sharpness for the commutator of the Riesz
transform when $p=2\,.$ Since $R^{\ast}_j=-R_j$, one can easily
check that the commutator of Riesz transforms are also
self-adjoint operators. Furthermore, choosing weight
$w(x)=|\,x|^{(n-\delta)(p-1)},$ we will obtain the sharpness for
$1<p<\infty$ by the same argument we used in the case of the
Hilbert transform.

\section{Acknowledgments}
This work is part of the author's Ph.D dissertation \cite{DW}. The
author like to thank Professor Carlos P\'{e}rez for his comments
and useful suggestions. Finally, the author is also very grateful
to his graduate adviser Mar\'{\i}a Cristina Pereyra for her
suggestions and helpful interaction.

\bibliographystyle{amsplain}

\noindent \textsc{Department of Mathematics and Statistics, 1 University of New Mexico, Albuquerque, NM 87131-0001}, \emph{E-mail:} dwchung@unm.edu
\textrm{ and } chdaewon@gmail.com
\end{document}